\documentclass{memoir}

\usepackage{amssymb,amsmath, txfonts,mathrsfs}
\usepackage{anysize} 
\usepackage{graphicx}
\makeindex

\input xy  \xyoption{all}

\newtheorem{deff}{Definition}[section]
\newtheorem{lemma}[deff]{Lemma}
\newtheorem{properties}[deff]{Properties}
\newtheorem{theorem}[deff]{Theorem}
\newtheorem{notation}[deff]{Notation}
\newtheorem{corollary}[deff]{Corollary}

\newtheorem{proposition}[deff]{Proposition}

\newtheorem{em-example}[deff]{Example}
\newtheorem{em-def}[deff]{Definition}        
\newtheorem{em-remark}[deff]{Remark}         
\newtheorem{em-question}[deff]{Question}

\newenvironment{example}{\begin{em-example} \em }{ \end{em-example}}
\newenvironment{definition}{\begin{em-def} \em  }{ \end{em-def}}
\newenvironment{remark}{\begin{em-remark} \em }{\end{em-remark}}
\newenvironment{question}{\begin{em-question}\em }{\end{em-question}}
\newenvironment{proof}{\noindent {\textit Proof}:}{\QED \smallskip}

\newcommand{\tr}{^{\triangleright}}

\newcommand\QED{\hfill QED \medskip}

\def\ker{\textrm{ker}}
\def\CHom{\textrm{CHom}}

\def\id{\textrm{id}}
\def\Hom{\textrm {Hom}}

\def\:{\nobreak \hskip .1111em\mathpunct {}\nonscript \mkern
-\thinmuskip {:}\hskip .3333emplus.0555em\relax}

\catcode`\@=12
\def\T{{\mathbb T}}
\def\Q{{\mathbb Q}}

\def\Z{{\mathbb Z}}
\def\N{{\mathbb N}}
\def\R{{\mathbb R}}
\def\Q{{\mathbb Q}}

\def\C{{\mathbb C}}

\title{Book's title}
\date{Date}
\author{Author's name}

\begin{document}



\Large{
 
\noindent    \centerline{ }

\noindent    \centerline{ }

\noindent    \centerline{ }

\noindent    \centerline{ }

\centerline{Proyecto Fin de M\'aster en Investigaci\'on Matem\'atica}

\centerline{Facultad de Ciencias Matem\'aticas, Universidad Complutense de Madrid}

\noindent   \centerline{ }

\noindent    \centerline{ }

\begin{figure}
  \centering
    \includegraphics[width=0.2\textwidth]{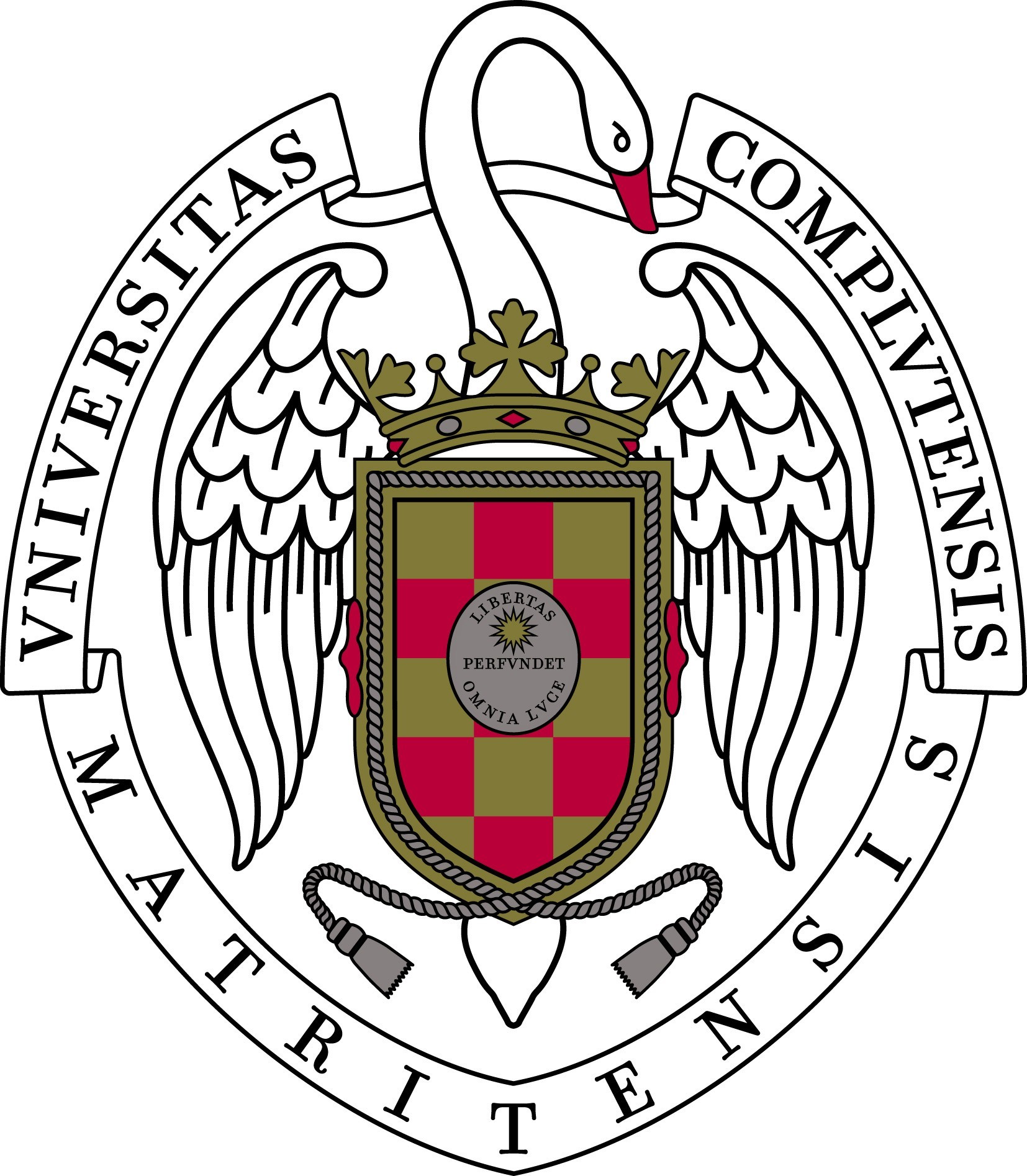}
\end{figure}

\noindent    \centerline{ }

\noindent    \centerline{ }

\noindent    \centerline{ }

\noindent    \centerline{ }

\centerline{{\huge {\textit{Locally quasi-convex topologies on the group of the integers}}}}

\noindent  \centerline{ }

\noindent   \centerline{ }

\noindent    \centerline{ }

\noindent    \centerline{ }

\noindent    \centerline{ }

\noindent    \centerline{ }

\noindent    \centerline{ }

\noindent    \centerline{ }

\centerline{Daniel de la Barrera Mayoral}

\centerline{Dirigido por: Elena Mart\'in Peinador}

\centerline {Curso acad\'emico: 2009/2010}

\newpage 

\noindent El abajo firmante, matriculado en el M\'aster en Investigaci\'on Matem\'atica de la Facultad de
Ciencias Matem\'aticas, autoriza a la Universidad Complutense de Madrid (UCM) a difundir y utilizar
con fines acad\'emicos, no comerciales y mencionando expresamente a su autor el presente Trabajo
Fin de M\'aster: "Locally quasi-convex group topologies on the group of the integers", realizado durante
el curso acad\'emico 2009-2010 bajo la direcci\'on de Elena Mart\'in Peinador, y la colaboración externa de Lydia Aussenhofer, en el Departamento de
Geometr\'ia y Topolog\'ia, y a la Biblioteca de la UCM a depositarlo en el Archivo Institucional EPrints
Complutense con el objeto de incrementar la difusi\'on, uso e impacto del trabajo en internet
y garantizar su preservaci\'on y acceso a largo plazo.

\noindent    \centerline{ }

\noindent    \centerline{ }

\noindent    \centerline{ }

\noindent    \centerline{ }

\noindent    \centerline{ }

\noindent    \centerline{ }

\noindent FDO: DANIEL DE LA BARRERA MAYORAL

\newpage 

\noindent \underline{Resumen}

\noindent La topolog\'ia m\'as natural en el grupo $\Z$ de los enteros, es la discreta. Son tambi\'en muy conocidas las topolog\'ias $p$-\'adicas en $\Z$, para cualquier n\'umero primo $p$. Otra topolog\'ia de grupo importante es la topolog\'ia d\'ebil asociada al grupo de los homomorfismos de $\Z$ en el c\'irculo unidad del plano complejo; es decir, la definida por los caracteres y que se conoce por "toplog\'ia de Bohr" en $\Z$.

\noindent En \cite{tesislorenzo}, se comprueba que tomando como base de entornos de cero los conjuntos $(W_n)$, donde $W_n:=\{k\in\Z\mid\forall x\in S$, $k\cdot x\in[-\frac{1}{4n},\frac{1}{4n}]+\Z\}$ con $S$ una sucesi\'on quasi-convexa en $\T$, se obtiene una toplog\'ia de grupo sobre los enteros, $\tau_S$. Sabemos que la topolog\'ia $\tau_S$ es metrizable y localmente cuasi-convexa.

\noindent En este trabajo caracterizamos las sucesiones convergentes en $\tau_S$, para ciertos subconjuntos $S\subset\T$. Damos condiciones sobre los elementos de $\Z$ para pertenecer a un cierto entorno $W_n$. Tomando como referencia una sucesi\'on $(b_n)$ de n\'umeros naturales, con ciertas restricciones, hemos considerado la "topolog\'ia lineal asociada" cuya base de entornos de $0$ es $\{b_n\Z\mid n\in\N\}$ y la de "la convergencia uniforme" en $S=\{\frac{1}{b_n}+\Z\mid n\in\N\}\subset \T$. Hacemos un estudio comparativo entre ambas clases de topolog\'ias.

\noindent    \centerline{ }

\noindent \underline{Palabras clave:} Topolog\'ia de grupo, subconjunto quasi-convexo, polar, dualidad, topolog\'ia lineal, car\'acter, topolog\'ias compatibles, topolog\'ia de la convergencia uniforme, topolog\'ia compacto-abierta.

\centerline{ }

\centerline{ }

\noindent \underline{Abstract}

\noindent The most natural group topology on $\Z$ is the discrete one. There are other well-known group topologies on $\Z$, like the $p$-adic, defined for any prime number $p$. It is also an important group topology the weak topology with respect to the group of homomorphisms from $\Z$ to the unit circle of the complex plane; that is, the one defined by the characters and which is known as "the Bohr topology" on $\Z$.

\noindent In \cite{tesislorenzo}, it is proved that taking as a neighbourhood basis at $0$ the subsets $\{W_n\mid n\in\N\}$, defined by $W_n:=\{k\in\Z\mid\forall x\in S$, $ k\cdot x\in[-\frac{1}{4n},\frac{1}{4n}]+\Z\}$, where $S$ is a quasi-convex sequence in $\T$, a group topology on $\Z$ is obtained, $\tau_S$. We know that the topology $\tau_S$ is metrizable and locally quasi-convex.

\noindent In this monograph we characterize convergent sequences in $\tau_S$, for some $S\subset\T$. We give sufficient conditions on the elements of $\Z$ in order that they belong to a neighbourhood $W_n$. For a fixed sequence $(b_n)$ of natural numbers, restricted to mild conditions, we have considered the "linear topology associated" whose neighbourhood basis at $0$ is $\{b_n\Z\mid n\in\N\}$ and the "topology of uniform convergence" on $S=\{\frac{1}{b_n}+\Z\mid n\in\N\}\subset\T$. We make a comparative study between both classes of topologies.

\noindent    \centerline{ }


\noindent    \centerline{ }


\chapter*{Introduction}
\noindent The idea of a general theroy of continuous groups is due to Sophus Lie, who developed his theory in the decade 1874-1884. Lie's work is the origin of both modern theory of Lie groups and the general theory of topological groups. However, the topological considerations that are nowadays essential in both theories, are not part of his work.

\noindent A topological point of view in the theory of continuous groups was first introduced by Hilbert. Precisely, in his famous list of 23 problems, presented in the International Congress of Mathematicians of 1900, held in Paris, the Fifth Problem boosted investigations on topological groups.

\noindent In modern language, the Fifth Problem asked if any locally euclidean topological group can be endowed with a structure of analytic variety in such a way that it becomes a Lie group. 

\noindent In 1929 von Neumann using integration on compact general groups, introduced by himself, solved the problem for compact groups. In 1934, Pontryagin solved it for locally compact abelian groups, using the character theory introduced by himself.

\noindent The final resolution, at least in this interpretation of what Hilbert meant, came with the work of A. Gleason, D. Montgomery and L. Zippin in the 1950s.

\noindent In 1953, Hidehiko Yamabe obtained the final answer to Hilbert's Fifth Problem: A connected locally compact group $G$ is a projective limit of a sequence of Lie groups, and if $G$ "has no small subgroups", then $G$ is a Lie group.

\noindent In 1932, Stefan Banach, defined in Th\`{e}orie des Op\`{e}rations Lin\`{e}aires, the spaces that would be named after him as special cases of topological groups. Since then, both theories have been developed in a different, but somehow parallel, way.

\noindent Several theorems for Banach, or even for locally convex spaces have been reformulated for abelian topological groups \cite{completeness}, \cite{eberlein-smulian}, \cite{bohrcompactification}, \cite{Schurproperty}, \cite{Dunford-Pettis}. The main obstacle for this task is the lack of the notion of convexity for topological groups. However, Vilenkin gave the definition of quasi-convex subset for a topological abelian group inspired in the Hahn-Banach theorem. In order to deal with it, many convenient tools had to be developed and this opened the possibility of a fruitful treatment of topological groups. The notion of quasiconvexity depends on the topology, in contrast with convexity, which is a purely algebraic notion.

\noindent After this, with quasi-convex subsets at hand, it was quite natural to define locally quasi-convex groups, which was also done by Vilenkin in \cite{vilenkin}.

\noindent Duality theory for locally convex spaces was mainly developed in the mid of twentieth century, and by now is a well-known rich theory.

\noindent There is a very natural way to extend it to locally quasi-convex abelian groups. From now on, we only speak of abelian groups although not explicitly mentioned. Fix first the dualizing object as the unit circle of the complex plane $\T$. The continuous homomorphisms from a topological group into $\T$ play the role of the continuous linear forms and they will be named continuous characters.

\noindent The set of continuous characters defined on a group $G$ has a natural group structure provided by the group structure of $\T$. Thus, we can speak of the group of characters of $G$. If it is endowed with the compact-open topology, the topological group obtained is called the dual group of $G$.

\noindent The research we have done in this memory deals with the following problem: Let $(G,\tau)$ be an abelian topological group and let $G^\wedge$ be its dual group. Consider the set of all locally quasi-convex group topologies in $G$ whose dual group coincides with $G^\wedge$. They will be called compatible topologies.

\noindent It is known that there is a minimum for this set, which is the weak topology induced by $G^\wedge$. 

\noindent However, it is not known if this set has a maximum element; whenever it exists, it will be called the Mackey topology for $(G,\tau)$. The problem of finding the largest compatible locally quasi-convex topology for a given topological group $G$, with this degree of generality (i.e. in the framework of locally quasi-convex groups) was first settled in 1999 in \cite{mackey}. Previously in 1964 Varopoulos \cite{varopoulos} had studied the question for the class of locally precompact group topologies.

\noindent The current state of the question is as follows: in general, it is not known if there is a maximum for the set of all locally quasi-convex topologies in $G$. Partial answers are the following: if $G$ is complete and metrizable, then there exists the Mackey topology \cite{mackey}. If $G$ is metrizable (but not complete) the original topology may not coincide with the Mackey topology \cite{unpublished}. 

\noindent It is known that there is always at least one locally quasi-convex compatible topology, namely the Bohr topology. This happens to be minimum. If $G$ is a locally compact abelian group, then the original topology coincides with the Mackey topology. Furthermore, the set of all locally quasiconvex compatible group topologies has cardinal greater or equal to $3$ \cite{unpublished}. In our future work we are looking forward to finding the general solution of existence (or non-existence) of the Mackey topology, as posed in \cite{mackey}. The conjecture is that it does not exist, in general, and the solution will pass through describing some topologies on $\Z$, the group of the integers, which may give us the same dual, but they may have no maximum (or the maximum generated by them, may not be compatible). So we start in this monograph by studying thoroughly different topologies on $\Z$.


\noindent We now explain the contents of the memory: In the first three chapters we introduce the fundamental notion of a topological group, give a survey of the main known result concerning topological groups, duality and reflexivity. Chapters 4 and 5 contain the \textbf{new} results obtained.

\noindent We summarize the contens of each chapter.

\noindent The basic definitions of topological groups, as well as general results on them and standard notation on duality are given in \textbf{chapter 1}.

\noindent \textbf{Chapter 2} deals with the relationship between duality and quasi-convex groups. It is important to observe that any dual group is locally quasi-convex, and, hence any reflexive group is locally quasi-convex. 

\noindent In section 5, we provide some results on embeddings of locally quasi-convex groups in a product of locally quasi-convex metrizable groups, which were stablished by Lydia Aussenhofer in \cite{tesislydia}.

\noindent In \textbf{chapter 3}, we tackle the group of the integers with different group topologies. First, we study the $2$-adic topology, denoted ussually by $\tau_2$, as a particular case of $p$-adic topologies. 

\noindent Then, we study topologies of uniform convergence on a \textbf{fixed} subset $S$ of $\T$. Here $\T$ is considered as the dual group of the integers. Clearly, the mentioned topology (we shall denote it by $\tau_S$) depends on the set $S$. We have started by taking $S$ to be a sequence. We have in mind to continue with other types of subsets $S$.

\noindent In this chapter, we describe the convergent sequences in $\tau_2$ and $\tau_S$ (for particular sets $S$), and as a consequence of the criterion found we can claim that the topologies $\tau_S$ and $\tau_2$ are related.

\noindent In \textbf{chapter 4}, we try to obtain the dual group of $\Z$ endowed with $\tau_S$ (for particular cases).

\noindent We remark the result obtained in \ref{bn}. It provides us with a development of an integer number as a sum pivoted by a particular -previously fixed- sequence of natural numbers. This tool, which is close to Number Theory, gives a constructive way to obtain the coefficients for this sum, and it is important is the sequel.

\noindent We characterize null sequences in $\tau_S$ in terms of the mentioned development. We also try to characterize the elements of a fixed neighbourhood at $0$ for $\tau_S$ through their development as a series pivoted by the elements in $S$. 

\noindent In theorem \ref{tau2noestauS}, we prove that for our particular type of set $S$, the topology $\tau_S$ does not coincide with the $2$-adic topology.

\noindent In \textbf{chapter 5}, we consider linear topologies on $\Z$, which are those whose neighbourhood basis at $0$ consists of subgroups of $\Z$.

\noindent For a fixed sequence $(b_n)$ of natural numbers we assign two different topologies, namely the linear topology whose neighbourhood basis is $\{b_n\Z\mid n\in\N\}$, denoted by $\tau_{(b_n\Z)}$, and the "associated $S$-topology generated by $(b_n)$" which is the topology of the uniform convergence on $\{\frac{1}{b_n}+\Z\mid n\in\N\}$, where $\T$ is considered to be the quotient $\R/\Z$. The latter will be denoted (as usual) by $\tau_S$.

\noindent We characterize the null sequences of $\tau_S$
and the elements of a fixed $0$-neighbourhood. 

\noindent We do also characterize the null sequences in $\tau_{(b_n\Z)}$.

\noindent Finally, we prove that for any $(b_n)$, the linear topology and the associated $S$-topology generated by $(b_n)$ are different.

\noindent These linear group topologies are in certain sense the natural extensions of the $p$-adic topology. In this line, the dual of $\Z$ endowed with these topologies is a natural extension of the Pr\"ufer's group.

\newpage
\tableofcontents 

\chapter{Preliminaries}
\section{General definitions and results on topological groups}

\begin{definition}
Let $G$ be the supporting set of a group and of a topological structure. Suppose that:

(i) the mapping $(x,y)\mapsto xy$ of $G\times G\rightarrow G$ is a continuous mapping.

(ii) the mapping $x\mapsto x^{-1}$ of $G\rightarrow G$ is continuous.

\noindent Then $G$ is called a topological group.
\end{definition}

\noindent \textbf{We will deal only with abelian groups.}

\noindent For  a topological group $G$,   the traslation mapping
$t_a:G\rightarrow G$ defined by $t_a(x)=x\cdot a$ (for $a \in G$) is
a homeomorphism (See proposition 1 in \cite{morris}).

\noindent As a consequence of this fact, the topology of a topological group can be defined just giving a neighbourhood basis at the neutral element.

\begin{lemma}\label{ctg}
Let $G$ be a topological group and let $\mathcal{U}$ be a non-empty subset of the power set of $G$ satisfying:

(i) $e\in U$ for all $U\in\mathcal{U}$.

(ii) For all $U\in\mathcal{U}$ there exists $V\in\mathcal{U}$ such that $-V\subseteq U$.

(iii) For all $U\in\mathcal{U}$, there exists $V\in\mathcal{U}$ such that $V+V\subseteq U$.

(iv) For every pair $U,V\in\mathcal{U}$, there exists $W\in\mathcal{U}$ such that $W\subseteq U\cap V$.

\noindent Then there exists a unique group topology $\mathcal{O}$ on $G$ such that $\mathcal{U}$ is a neighbourhood basis at the unit element. $(G,\mathcal{O})$ is a Hausdorff space if and only if $\displaystyle{\bigcap_{U\in\mathcal{U}}U=\{e\}}$.
\end{lemma}

\begin{proof}
See chapter III, $\S 1.2$, proposition 1 in \cite{bourbaki}.
\end{proof}

\begin{lemma}\label{subgrupoabierto}
Let $G$ be a topological group. If a subgroup $H\subset G$ contains
a neighbourhood of the neutral element, then $H$ is open.
\end{lemma}

\begin{definition}
If $G$ is a group, and $S\subset G$, we denote by $\langle S\rangle$ the subgroup generated by $S$.
\end{definition}

\begin{definition}
A group topology is called linear if it has a neighbourhood basis at $0$ consisting of subgroups.
\end{definition}

\section{Basic definitions on duality}

\noindent In this section we describe the dual group of a
topological group. We omit  the  proofs which are simple
verifications.

\begin{notation}
The unit circle of the complex plane is denoted by  $\T:
=\{z\in\C:\mid z\mid=1\}$ and let $\T_+ : =\T\cap\{\mbox{Re}z\geq
0\}$
\end{notation}

\begin{remark}
We can also consider $\T$ as the quotient of $\R$ over $\Z$. In this way we shall consider $\T\approx[-\frac{1}{2},\frac{1}{2})$. We define $\T_+=[-\frac{1}{4},\frac{1}{4}]$ and $\T_m=[-\frac{1}{4m},\frac{1}{4m}]$ for any integer $m$.
\end{remark}

\begin{notation}
 Let $G$ be a topological group. The set of all continuous homomorphisms from $G$ to
 $\T$ will be denoted by $\CHom(G,\T)$.
 \end{notation}

 \noindent Clearly $\CHom(G,\T)$ endowed with the natural multiplication:
 $\varphi_1\varphi_2:x\mapsto\varphi_1(x)\varphi_2(x)$ is an abelian
 group.

\begin{deff}\label{definiciondual}
The dual of a topological group $G$ is the group
$\CHom(G,\tau)$ equipped with the compact open topology.
 We shall denote it by $G^\wedge$ (or by $(G,\tau)^\wedge$
it we want to stress the topology on $G$.)
\end{deff}

\begin{definition}
Let $(G,\tau)$ a topological group. The polar  of a subset $S\subset
G$ is defined as $S\tr: =\{\chi\in G^\wedge\mid\chi(S)\subset\T_+\}$.
\end{definition}

\begin{lemma}
Let $(G,\tau)$ be a topological group. The compact-open topology in
$G^\wedge$ can be described as the family of sets $\mathcal{U}_{G^\wedge}(1)=\{K\tr
\mid K\subset G \mbox{  is compact } \}$ taken as a neighborhood basis at $0$.
\end{lemma}

\noindent It is straightforward to prove that the compact open topology on
  $G^\wedge$   is in fact  a group topology.

\begin{definition}
For a dual group $G^\wedge$, the inverse  polar of a subset
$S\subset G^\wedge$  is defined as $S^{\triangleleft}: = \{x\in
G\mid \chi(x)\in\T_+ \mbox{ for all }\chi\in S\}$.
\end{definition}

\begin{definition}
Let $(G,\tau)$ be a topological group. We define $G^{\wedge\wedge}=(G^\wedge)^\wedge$ and endow it with the corresponding compact-open topology. That is, the topology is defined by $\mathcal{U}_{G^{\wedge\wedge}}(1)=\{K\tr\mid K\subset G^\wedge \mbox{ compact}\}$.
\end{definition}

\chapter{Duality and locally quasi-convex groups}

\noindent Duality and locally quasi-convexity are two related notions. As we will see, for any topological group $(G,\tau)$, its dual is locally quasi-convex.

\section{General results on duality}

\noindent We now state the main definitions and elementary results concerning duality of topological groups.

\noindent Here, instead of considering $\T$ as a subgroup of $\C$ we are going to consider $\T$ as the quotient of $\R$ over $\Z$.

\noindent The torsion subgroup of $\T$ is $\T_t=\{\xi_n^k\mid k\in\N_0, n\in\N\}=\{e^{2\pi i\frac{k}{n}}\mid n\in\N, k\in\{0,\dots n-1\}\}=\{e^{2\pi iq}\mid q\in\Q\}\approx \Q/\Z$.

\begin{proposition}
If $\beta\in \R\setminus\Q$ then $L=\langle e^{2\pi i \beta}\rangle$ is a dense subgroup in $\T$
\end{proposition} 

\begin{proof} Suppose that $L$ is a non dense subgroup of $\T$.

\noindent Consider the following mapping $\varphi:\Z\rightarrow\T$, where $n\mapsto e^{2\pi in\beta}$.

\noindent First, we show that $\varphi$ is injective. Suppose there exist $n,m\in\N$ such that $\varphi(n)=\varphi(m)$. This means that $e^{2\pi in\beta}=e^{2\pi im\beta}$. This implies that $n\beta-m\beta\in\Z$. Since $\beta\notin\Q$, we get that $n\beta-m\beta=0$; or equivalently, $n=m$.

\noindent Since, $\varphi$ is injective $L$ is an infinite subset of $\T$, and $p^{-1}(L)$ is not dense.

\noindent We consider now $p:\R\rightarrow \T$. $p^{-1}(L)=gp\langle \beta,1\rangle$. We know that $p^{-1}(L)$ is closed (proposition 19 in \cite{morris}).

\noindent Choose $a,b\in\R$. We see, now, that $gp\rangle a,b\langle$ is closed if and only if $a$ and $b$ are rationally dependent.

\noindent The only closed subgroups in $\R$ are $\R$, the empty set and those isomorphic to $\Z$ (proposition 20 in \cite{morris}). Hence $gp\langle a, b\rangle$ is closed implies $gp\langle a, b\rangle\approx c\Z$. Obviously, $a=z_1c$ and $b=z_2c$ and they are rationally dependent.

\noindent Conversely, suppose that $a=\frac{m}{n}b$, where mcd$(m,n)=1$. Hence therer exist $p,q$ such that $mp+nq=1$. Multiplying by $a$, we get $amp+anq=a$. Equivalently, $amp+mbq=a=m(ap+bq)$. Multiplying $mp+nq=1$ by b we get $b=n(ap+bq)$. Hence $gp\langle a,b\rangle=gp\langle ap+bq\rangle\approx\Z$, and as consequence closed.

\noindent Since our $p^{-1}(L)=gp\langle\beta,1\rangle$ is closed $1$ and $\beta$ are rationally dependent. Then $\beta\in\Q$, which contradicts the hypothesis.

\end{proof}


\noindent Another easy result is: 

\begin{proposition}
If $G$ has order $m$, then $G^\wedge$ has order $m$.
\end{proposition}

\begin{proof}
$\chi^m(x)=\stackrel{(m)}{\chi(x)\cdots \chi(x)}=\chi(mx)= \chi(0)=1$ for all $\chi\in G^\wedge$.
\end{proof}

\noindent In $G^\wedge$ we consider the compact-open topology, $\tau_{CO}$. We define the compact-open topology in the following way: $V_{K,U}=\{\varphi\in\CHom(G,\T)\mid \varphi(K)\subset U\}$. We take as subbasis for the compact-open topology the family $V_{K,U}$ where $K\subset G$ is compact and $U$ is an open subset of $\T$.

\noindent It can be easily seen that the family $\{K^\triangleright\mid K\subset G\mbox{ is compact}\}$ is a neighbourhood basis at $1$ for $\tau_{CO}$.

\noindent In particular for $G=\R_u$, every compact $K\subset\R$ is contained in $[-n,n]$ for a suitable natural number, in addition, $[-n,n]^\triangleright \subset K^\triangleright$. Therefore a neighbourhood basis at the null character $1\in\R^\wedge$ is given by the sets: $[-n,n]\tr=\{\chi_t\mid \chi_t(x)\in\T_+$ $\forall \mid x\mid \leq n\}= \{\chi_t\mid e^{2\pi itx}\in\T_+$ $\forall \mid x\mid\leq n\}=[-\frac{1}{4n},\frac{1}{4n}]$.

\noindent Thus, $\R^\wedge$ can be algebraical and topologically identified with $(\R,\tau_u)$


\noindent We also know that $\T^\wedge=\Z$ and $\Z^\wedge=\T$


\noindent Clearly: a topological group is discrete if and only if $\{0\}$ is open.

\noindent The first statement about duality is:

\begin{proposition}
\noindent Let $K$ be a compact abelian topological group. Then $K^\wedge$ is discrete. 
\end{proposition}

\begin{proof}
$K\tr=\{\chi\in K^\wedge\mid\chi(K)\subseteq\T_+\}$. Since $\chi(K)$ is a subgroup of $\T$, $\chi(K)\subset\T_+ \Longleftrightarrow \chi(K)=\{1\}$. Hence $K\tr=\{\chi\in K^\wedge\mid\chi(K)=\{1\}\}$. Thus, $K\tr$ consists only of the null character, $1$. Hence $\{1\}$ is open, and $K^\wedge$ is discrete. 
\end{proof}

\begin{proposition}
Let $(G,\tau_{dis})$ be a discrete topological group. Then $\Hom(G,\T)$ is a compact subgroup of $\T^G$.
\end{proposition}

\begin{proof}
\noindent By Tychonoff theorem, we know that $\T^G$ is compact. Therefore, it suffices to prove that $\Hom(G,\tau)$ is a closed subgroup of $\T^G$.

\noindent In order to prove that $\Hom(G,\T)$ is closed in $\T^G$, we take $(\chi_j)_{j\in J}$ a net in $\Hom(G,\T)$ converging to $\chi:G\rightarrow \T$. We must see that $\chi$ is a homomorphism.

\noindent Let $x,y\in G$. It suffices to show that $\chi(x+y)=\chi(x)\chi(y)$. 

\noindent $\chi(x+y)=\lim_j\chi_j(x+y)=\lim_j(\chi_j(x)\chi_j(y))=\lim_j\chi_j(x)\lim_j\chi_j(y)=\chi(x)\chi(y)$

\noindent Hence $\Hom(G,\T)$ is closed.

\end{proof}

\begin{corollary}
Let $(G,\tau)$ be a discrete group. Then $(G,\tau)^\wedge$ is compact.
\end{corollary}

\begin{proof} 
\noindent Since $\Hom(G,\tau)=(G,\tau)^\wedge$ the result follows.

\end{proof}


\begin{definition}
The natural embedding $\alpha_G:G\rightarrow G^{\wedge\wedge}$, is defined by $x\mapsto\alpha_G(x):G^\wedge\rightarrow\T$, where $\chi\mapsto\chi(x)$
\end{definition}

\begin{proposition}
$\alpha_G$ is a homomorphism.
\end{proposition}

\begin{proof}
$\alpha_G(x+y)(\chi)=\chi(x+y)=\chi(x)\chi(y)=\alpha_G(x)(\chi)\alpha_G(y)(\chi)= (\alpha_G(x)\alpha_G(y))(\chi)$, for all $\chi\in G^\wedge$.

\noindent Thus, $\alpha_G(x+y)=\alpha_G(x)\alpha_G(y)$

\end{proof}

\begin{definition}
A topological group $(G,\tau)$ is reflexive if $\alpha_G$ is a topological isomorphism
\end{definition}

\begin{proposition}
The canonical mapping $\alpha_G$ is injective if and only if $G^\wedge$ separates points; that is, for every $x\neq 0$, there exists $\chi\in G^\wedge$ such that $\chi(x)\neq 1$.
\end{proposition}

\noindent The following highly non-trivial assertion was proved by Weyl and is the corner-stone for the duality Theorem of Pontryagin-Van Kampen.

\begin{theorem}
Let $G$ be a compact abelian topological group, then $\alpha_G$ is injective.
\end{theorem}

\noindent An important tool in our subsequent work is the Pontryagin-Van Kampen theorem, which states the following.

\begin{theorem} [Pontryagin-Van Kampen]\label{PVK} Let $(G,\tau)$ be a locally compact abelian topological group. Then $G$ is reflexive.
\end{theorem}

\section{On the continuity of $\alpha_G$ and equicontinuity}

\noindent The general definition of equicontinuity in the framework of uniform spaces and continuous mappings can be restated in a simpler form in the particular case of topological groups and continuous homomorphisms as we do next:

\begin{definition}
Let $S\subset G^\wedge$, $S$ is equicontinuous if for every $W\in\mathcal{U}_\T(0)$ there exists $V\in\mathcal{U}_G(0)$ such that $\varphi(V)\subset W$ for every $\varphi\in S$.
\end{definition}

\noindent $\alpha_G$ is continuous if and only if for every compact $K\subset G^\wedge$, $\alpha_G^{-1}(K\tr)$ is a neighbourhood of $0$ in $G$.

\noindent $\alpha_G^{-1}(K\tr)=\{x\in G\mid \alpha_G(x)\in K\tr\}=\{x\in G\mid\alpha_G(x)(\chi)\in\T_+,\ \forall\chi\in K\}=\{x\in G\mid\chi(x)\in\T_+,\ \forall\chi\in K\}$.

\noindent We have that $\alpha_G$ is continuous if and only if for every $K\subset G^\wedge$ compact, there exists $U\in\mathcal{U} _G(0)$ such that  $U\subset\{x\mid \chi(x)\in\T_+,\ \forall\chi\in K\}$

\noindent That is to say, every compact of $G^\wedge$ is equicontinuous.

\noindent A further simplification is given in the following proposition.

\begin{proposition}
$S\subset G^\wedge$ is equicontinuous if there exists $U\in\mathcal{U}_G(0)$ such that for every $x\in U$ and for every $\chi \in S$, $\chi(x)\in\T_+$, that is $S\subset U^\triangleright$. 
\end{proposition} 

\begin{proof}
\noindent Let $[-\frac{1}{4n},\frac{1}{4n}]+\Z$ be a neighbourhood of $0+\Z$ in $\T$.

\noindent By \ref{ctg} (iii), there exists $V$ where $V+\stackrel{(n)}{\cdots} +V\subset U$. Hence, $\chi(V+\stackrel{(n)}{\cdots}+V)\subset\chi(U)\subset\T_+$.

\noindent Choose $x\in V$; $\chi(x),\chi(2x),\dots ,\chi(nx)\in\T_+\Rightarrow \chi(x)\in[-\frac{1}{4n},\frac{1}{4n}]+\Z\Rightarrow \chi(V)\subset [-\frac{1}{4n},\frac{1}{4n}]+\Z$.

\noindent $U\tr\supset S\Longleftrightarrow \forall n,\ \exists V$ such that $\forall\chi\in S,\ \forall x\in V$ $\chi(x)\in[-\frac{1}{4n},\frac{1}{4n}]+\Z$

\end{proof}

\begin{proposition}
If $G$ is metrizable, then $\alpha_G$ is continuous.
\end{proposition}

\section{Introduction to locally quasi-convex groups}

\noindent The notion of convexity is one of the most fruitful tools in Mathematics. A convex set is defined only in the context of vector spaces in a pure algebraic mode.

\noindent For a topological group $G$ the notion of convexity is not available, since there is no scalar multiplication. However an analogous definition could be obtained modeled in the separation theorem, which is obtained as a corollary to Hahn-Banach theorem, which states the following: let $E$ be a locally convex topological vector space. A subset $M\subset E$ is convex if for every $x\in E\setminus M$, there exists a continuous linear mapping $\varphi$ such that $\varphi (x)>\varphi(y)$ for every $y\in M$.

\noindent For a topological group $(G,\tau)$ the analogous definition would rely on  continuous homomorphisms, instead of linear forms.

\noindent However, the dualizing object cannot be $\R$ as we make evident in the following paragraph: 

\noindent Let $G$ be a compact group and let $\varphi:G\rightarrow\R$ be a continuous group homomorphism. Since $G$ is compact and $\varphi$ is a continuous homomorphism, $\varphi(G)$ is a compact subgroup of $\R$. Hence $\varphi(G)=\{0\}$.

\noindent We must choose another group instead of $\R$. In 25.36 of \cite{hewitt-ross} it is proved that for results on duality, this group must be $\T$. Hence we pick characters (that is continuous group homomorphisms into $\T$) $\chi:G\rightarrow \T$, instead of continuous linear forms.

\noindent Now, there exists no order in $\T$, so it has no sense to state $\chi(x)<\chi(y)$. The "modified definition" given by Vilenkin in the 50's, see \cite{vilenkin}, goes as follows:

\begin{definition}
Let $(G,\tau)$ be a topological group. A subset $S$ is quasi-convex if for every $x\in G\setminus S$ there exists $\chi\in G^\wedge$ such that $\chi(S)\subset\T_+$ and $\chi(x)\notin\T_+$, where $G^\wedge=\{\chi\mid \chi:G\rightarrow \T\mbox{ continuous}\}$ and $\T_+=\T \cap \{Re(z)\geq 0\}$.
\end{definition}

\begin{definition}
A topological group $G$ is said to be locally quasi-convex if the neutral element $e_G$ has a neighbourhood basis formed by quasi-convex sets.
\end{definition}

%



\section{Properties of quasi-convex sets}

\noindent In this section we prove some properties of quasi-convex sets, and also we prove that the dual of any topological group is locally quasi-convex. This fact will be a easy prove of the fact that any reflexive toplogical group is locally quasi-convex.

\begin{properties}

\noindent Let $(G,\tau)$ be a topological group and $A\subset G$ be a quasi-convex set, then:

(i) $A$ is symmetric. 

(ii) $0\in A$.

(iii) $\displaystyle{A=\bigcap_{\chi\in A\tr}\chi^{-1}(\T_+)=A^{\triangleright\triangleleft}}$

(iv) $A$ is closed.

\end{properties}

\begin{proposition}
Let $(G,\tau)$ be a topological group and let $M\subset G$. Then $M\tr$ is quasi-convex
\end{proposition}

\begin{proof}
Take $\chi_0\in G^\wedge\setminus M\tr$; that is, $\exists x\in M$ such that $\chi_0(x)\notin\T_+$. $\alpha_G(x)\in G^{\wedge\wedge}$, and for all $\chi\in M\tr$ we have $\alpha_G(x)(\chi)=\chi(x)\in\T_+$ but $\alpha_G(x)(\chi_0)=\chi_0(x)\notin\T_+$

\end{proof}

\begin{proposition}
\noindent Let $(G,\tau)$ be a topological group; then $(G^\wedge,\tau_{CO})$ is locally quasi-convex.
\end{proposition}

\begin{proof}
We remember that the family $\mathcal{K}=\{K\tr\mid K \mbox{ is compact}\}$ is a neighbourhood basis for $\tau_{CO}$. 

\noindent We have just proved that the polar of any subset is quasi-convex, in particular $\mathcal{K}$ is a neighbourhood basis formed by quasi-convex sets.

\noindent Hence, $(G^\wedge,\tau_{CO})$ is locally quasi-convex.

\end{proof}

\begin{example}
\noindent $\R$ is locally quasi-convex, since $[-\frac{1}{n},\frac{1}{n}]$ is quasi-convex for all $n$.

\noindent $\T$ is locally quasi-convex, since $\{[-\frac{1}{4n},\frac{1}{4n}]+\Z\mid n\in\N\}$ is a quasi-convex neighbourhood basis for $0$ in $\T$. In particular, $\T_+$ is quasi-convex.
\end{example}

\begin{proposition}
Let $\varphi:(G,\tau)\rightarrow(H,\sigma)$ be a continuous homomorphism and $A\subset H$ quasi-convex; then $\varphi^{-1}(A)\subset G$ is quasi-convex.
\end{proposition}

\begin{proof}
\noindent Let $x\notin\varphi^{-1}(A)$, then $\varphi(x)\notin A$.

\noindent Since $A$ is quasi-convex, there exists $\chi\in H^\wedge$ such that $\chi\varphi(x)\notin\T_+$ and $\chi(A)\subset\T_+$. Consider $\Psi=\chi\circ\varphi:G\rightarrow\T$.

\noindent $\Psi(\varphi^{-1}(A))=\chi\circ\varphi\circ\varphi^{-1}(A)\subset\chi(A)\subset\T_+$.

\noindent $\Psi(x)=\chi(\varphi(x))\notin\T_+$.

\end{proof}

\begin{proposition}
Let $(A_i)_{i\in I}\subset G$ be a family of quasi-convex sets. Then $\cap_{i\in I}A_i$ is quasi-convex.
\end{proposition}

\begin{proof}
\noindent $\displaystyle{x\notin\bigcap_{i\in I}A_i\Rightarrow}$ $\exists i_0$ such that $x\notin A_{i_{0}}$.

\noindent Since $A_{i_{0}}$ is quasi-convex, there exists $\varphi\in G^\wedge$ such that $\varphi(x)\notin\T_+$ and $\varphi(A_{i_{0}})\subset\T_+$.

\noindent Since $\displaystyle{\bigcap_{i\in I}A_i\subset A_{i_{0}}}$, $\displaystyle{\varphi(\bigcap_{i\in I}A_i)\subset\varphi(A_{i_{0}})\subset\T_+}$.

\end{proof}

\begin{proposition}
$S\subset G^\wedge$. Then $\displaystyle{S^\triangleleft:=\bigcap_{\chi\in S}\chi^{-1}(\T_+)}$ is quasi-convex.
\end{proposition}

\begin{proof}
$\T_+$ is quasi-convex $\Rightarrow \chi^{-1}(\T_+)$ is quasi-convex $\displaystyle{\Rightarrow \bigcap_{\chi\in S}\chi^{-1}(\T_+)}$ is quasi-convex.
\end{proof}

\section{Embedding of a  locally quasi-convex group into a product of metrizable locally quasi-convex groups}
\noindent Our next aim is to embed a Hausdorff locally quasi-convex topological group into a product of groups with "good" properties.

\noindent A deeper result was obtained by Lydia Aussenhofer in \cite{tesislydia}, but on the other hand it is much more difficult to prove because it also deals with dually closed images. We shall omit the result that the resulting group are complete, getting a weaker result.

\noindent Furthermore, this proof tries to be self-contained using only results proved in this section 

\begin{remark}
In this section we shall consider $(G,\tau)$ a topological group as usual and  $U$ to be a quasi-convex neighbourhood at $0$ for $\tau$. We must note that some of the results of this section are also true for a generic subset. 
\end{remark}

\begin{deff}
Let $X\subset G$. Define $\left(\frac{1}{n}\right)X=
\{x\in G|\ x,2x,\dots,nx\in X\}$ and $X_\infty=\displaystyle{\bigcap_{n\geq 1}\left(\frac{1}{n}\right)X}$. In particular, $\displaystyle\left(\frac{1}{n}\right)\T_+=\left[-\frac{1}{4n},\frac{1}{4n}\right]+\Z=\T_n$ and $\T_\infty=\{0\}+\Z$.
\end{deff}

\begin{lemma}\label{caracterizacion1/nU}
Let $(G,\tau)$ be a topological group and $U\subset G$ be a quasi-convex neighbourhood, then $\left(\frac{1}{n}\right)U=\displaystyle\bigcap_{\chi\in U\tr}\chi^{-1}(\T_n)$
\end{lemma}

\begin{proof}
$x\in\left(\frac{1}{n}\right)U  \Longleftrightarrow  kx\in U,  1\leq k\leq n  \Longleftrightarrow  \forall \chi \in U\tr ,  1\leq k\leq n, \chi(kx) \in\T_+  \Longleftrightarrow  \forall \chi\in U\tr , 1\leq k \leq n, k\chi(x)\in\T_+  \Longleftrightarrow  \forall \chi\in U\tr , \chi(x)\in\T_n  \Longleftrightarrow x\in\displaystyle\bigcap_{\chi\in U\tr}\chi^{-1}(\T_n)$

\end{proof}

\begin{lemma}
Let $(G,\tau)$ be a topological group, and $U\subset G$ a quasi-convex neighbourhood. Then $\left(\left(\frac{1}{n}\right)U\right)_{n\geq 1}$ is a neighbourhood basis at $0$ for a group topology.
\end{lemma}

\begin{proof}
\noindent We check conditions set in \ref{ctg} . Conditions (i) and (iv) are trivial from the choice we have done. (ii) is obtained as a direct consequence of the fact that $U_n$ is symmetric. In order to prove (iii), we will see that $\left(\frac{1}{2n}\right)U+\left(\frac{1}{2n}\right)U\subset \left(\frac{1}{n}\right)U$. Let $x\in\displaystyle\left(\frac{1}{2n}\right)U$ for all $\chi\in U\tr$. By the previous characterization, we have $\chi(x)\in\T_{2n},\ \forall \chi\in U\tr$. Let $x,y\in\left(\frac{1}{2n}\right)U$, therefore $\chi(x+y)=\chi(x)+\chi(y)\in\T_{2n}+\T_{2n}=\T_n,\ \forall\chi\in U\tr$. Equivalently, $x+y\in\left(\frac{1}{n}\right)U$.
\end{proof}

\noindent Then, we have that $\{(\frac{1}{n})U\mid n\in\N\}$ is a neighbourhood basis at $e$ for a group topology which will be denoted by  $\mathfrak{T}_U$.

\begin{lemma}
The set $U_\infty$ is a subgroup of $G$ for any $U$
\end{lemma}

\begin{proof}
$U_\infty =\displaystyle\bigcap_n \left(\frac{1}{n}\right)U=\displaystyle\bigcap_n \displaystyle
\bigcap_{\chi\in U\tr} \chi^{-1}(\T_n)= \displaystyle\bigcap_{\chi\in U\tr} \chi^{-1}(\cap_n \T_n)=
\displaystyle\bigcap_{\chi\in U\tr} \chi^{-1}(\{0+\Z\})= \displaystyle\bigcap_{\chi \in U\tr}\ker(\chi)< G$.

\end{proof}

\begin{lemma}
Let $(G,\tau)$ be a topological group. Then $\overline{\{0\}}^{\tau}=\displaystyle\bigcap_{V\in\mathcal{U}_\tau(0)}V$.
\end{lemma}

\begin{proof}
$G\setminus\overline{\{0\}}=\{x\in G|\ \exists W\in \mathcal{U}(x):\ 0\notin W\}$.
We can take $W=xV$ where $V\in \mathcal{U}(0)$ is symmetric. Now,
$0\notin xV \Longleftrightarrow x^{-1}\notin V \Longleftrightarrow x\notin V$.
Then, $\{x\in G|\ \exists W\in \mathcal{U}(x),\ 0\notin W\}=\{x|\  \exists V\in \mathcal{U}(0), \ x\notin V\}= G\setminus(\displaystyle\bigcap_{V\in\mathcal{U}_\tau(0)}V)$.

\end{proof}

\begin{lemma}
Let $(G,\tau)$ be a topological group. Then $\displaystyle{\overline{\{0\}}^{\mathfrak{T}_U}}=U_\infty$
\end{lemma}

\begin{proof}
\noindent $\displaystyle{\overline{\{0\}}^{\mathfrak{T}_U}=
\displaystyle\bigcap_{n} \left(\left(\frac{1}{n}\right)U\right)=U_\infty}$. It is enough to pick a neighbourhood basis, because any other neighbourhood contains a basic one.
\end{proof}

\begin{corollary}
Given a topological group $(G,\tau)$, $U_\infty$ is a closed subgroup of $G$.
\end{corollary}

\noindent By the characterization of $\left(\frac{1}{n}\right)U$ of \ref{caracterizacion1/nU}, we can define $\mathfrak{T}_U$ in another way. We choose as neighbourhoods: $V_{U\tr ,n}=\{x\in G\mid\forall\chi\in U\tr, \chi(x)\in\T_n\}=\displaystyle\bigcap_{\chi\in U\tr}\chi^{-1}(\T_n)=\left(\frac{1}{n}\right)U$.

\noindent Then, we have two possibilities to reach the same topology. Given $U$ a quasi-convex neighbourhood of $0$, we can take either the topology generated by $\left(\frac{1}{n}\right)U$, or the $U\tr$-topology (or topology of uniform convergence on $U\tr$).

\noindent We now consider the mapping

\centerline {$\id:(G,\tau)\rightarrow (G,\mathfrak{T}_U)$}

\noindent Since $\{(1/n)U\mid n\in\N\}$ is a neighbourhood basis at $0$
in $(G, \mathfrak{T}_U)$ and $(1/n)U$ is a neighbourhood of $0$ in $\tau$, the mapping is continuous.

\noindent We can now define

\centerline {$\varphi_U:(G,\mathfrak{T}_U)\rightarrow (G/U_\infty, \mathfrak{T}_U/U_\infty) $}

\noindent as the canonical projection. Since $U_\infty <G$, the quotient is well-defined. Considering now the composition of both mappings, we get:

\centerline {$\varphi_U\circ \ \id:(G,\tau) \rightarrow (G/U_\infty, \mathfrak{T}_U/U_\infty) $}

\noindent We define $(H_U,\sigma_U):=(G/U_\infty, \mathfrak{T}_U/U_\infty) $.

\noindent We consider $\chi\in U\tr$, we know $\chi(U_\infty)\subset\chi(U)\subset\T_+$.
Since $\chi(U_\infty)$ is a subgroup which is contained in $\T_+$, it must be $\chi(U_\infty)=\{0\}$.
Consequently, $\chi$ factorizes in a unique way described in the diagram:

\centerline{$\begin{matrix}G\overset{\chi}{\longrightarrow }\T\\ \quad\varphi_U\;{\searrow{\;\;}}
\nearrow{\overline{\chi}\quad}\\  G/U_\infty\end{matrix}$}

\noindent Since $\chi\left(\left(\frac{1}{n}\right)U+U_\infty\right)= \chi\left(\left(\frac{1}{n}\right)U\right)\subset\T_n$ holds, we get the continuity of $\bar{\chi}$. Furthermore, we can consider $U\tr$ as a subset in $(H_U,\sigma_U)\tr$.

\noindent We choose $V_{U\tr,n}=\{x+U_\infty\mid\forall\chi\in U\tr \ \bar{\chi}(x+U_\infty)=\chi(x)\in\T_n\}=\left(\frac{1}{n}\right)U+U_\infty$ as neighbourhood basis in $\mathfrak{T}_U/U_\infty=\sigma_U$.

\noindent We want that $\sigma_U=\mathfrak{T}_S$, for some adequate $S$. $\varphi_U(U)$ is a neighbourhood at $0$ in $(H,\sigma_U)$ because $U$ is a neighbourhood at $0$ in $(G,\mathfrak{T}_U)$ and the quotient mapping is always open.

\noindent The proof of the main result of this section is supported by the following two lemmas:

\begin{lemma}\label{lema1}
$U+\ker(\varphi_{U})=U$
\end{lemma}

\begin{proof}
Clearly $U+\ker(\varphi_U)\supset U$. We now prove the converse. Choose $x\in U, y\in U_\infty$; $\chi(x+y)=\chi(x)+\chi(y)=\chi(x)\in \T_+ \ $ for every $\chi \in U\tr$.

\end{proof}

\begin{lemma}\label{lema2}
Let $(G,\tau)$ be a Hausdorff locally quasi-convex topological group and let $U\in\mathcal{U}_\tau (0)$ be a quasi-convex set. Then, there exists $\varphi_U:(G,\tau) \rightarrow (H_U,\sigma_U)$ continuous, $\varphi_U(U)$ is an open neighbourhood in $(H_U,\sigma_U)$; furthermore $\sigma_U=\mathfrak{T}_S$ for $S=U\tr$ and $(H_U,\sigma_U)$ is locally quasi-convex.
\end{lemma}

\begin{proof}
It only remains to see that $(H_U,\sigma_U)$ defined above is locally quasi-convex.

\noindent Let $\varphi_U$ be the canonical projection; we will see that
$\{\varphi_U\\frac{1}{n}U\mid n\in\N\}$ is a neighbourhood basis for the topology we are interested in. For that it is enough to see that: 
$\varphi_U((\frac{1}{n})U) =V_{U\tr,n}$.
$V_{U\tr,n}=\{x+U_\infty \mid \forall \chi\in U\tr\ \overline{\chi}(x+U_\infty)\in\T_n\}=\{x+U_\infty\mid
x\in (\frac{1}{n})U\} =\varphi_U((\frac{1}{n})U)$

\end{proof}

\noindent We can now state and prove the main result:

\begin{theorem}
Let $(G,\tau)$ be a Hausdorff locally quasi-convex group. Then

\centerline{$\Psi:G\rightarrow \displaystyle\prod_{U\in\mathcal{U}(0)\ q.c}(H_U,\sigma_U)$}

\noindent where $x\mapsto(\varphi_U(x))_{U\in \mathcal{U}(0)}$ is a topological embedding. Furthermore $(H_U,\sigma_U)$ are metrizable groups.
 \end{theorem}

\begin{proof}
By the previous constructions $(H_U,\sigma_U)$ is Hausdorff, since it is a quotient by a closed subgroup. It is metrizable, since $\{(\frac{1}{n})U\mid n\in\N\}$ is a countable neighbourhood basis at $e$.

\noindent We will see that $\Psi$ is an embedding:

\noindent First we prove it is one-to-one. $x\in ker(\Psi)$; that is: $\varphi_U(x)=0 \ \forall U \Rightarrow x\in \displaystyle\bigcap_{U}ker(\varphi_U)\subset \displaystyle\bigcap_{U\in \mathcal{U} \ q.c}U=\overline{\{0\}}^\tau$. Since $\tau$ is Hausdorff, this implies that $x=0$. Hence $\Psi$ is one-to-one.

\noindent The continuity of each $\varphi_U$ is determined in \ref{lema2}. Since $\Psi$ is continuous in each factor (of a product), in particular, it is continuous.

\noindent Finally, we must see that $\Psi$ is open respect to its image.
It suffices to prove that for each quasi-convex neighbourhood $U$, $\Psi(U)$ is a neighbourhood in $\Psi(G)$.

\noindent $\Psi(G)\cap (\varphi_U(U) \times \displaystyle\prod_{V\neq U}H_V)=\{\Psi(x)\mid \varphi_U(x)\in \varphi_U(U)\}=\{\Psi(x)\mid x\in U+\ker(\varphi_U)\}$. By \ref{lema1}, the last term equals  $\{\Psi(x)\mid x\in U\}=\Psi(U)$.

\end{proof}

\chapter{Different topologies on the group of the integers}
\section{The $2$-adic topology}

\noindent The most natural non-discrete topologies on $\Z$ are the $p$-adic topologies. We first study the 2-adic topology in $\Z$, which is usually denoted by $\tau_2$.

\noindent The 2-adic topology is defined by the following neighbourhood basis at $0$: $\mathcal{U}=\{2^n\Z\mid n\in\N\}$. We will denote $2^n\Z$ by $U_n$. It is easy to check that is a neighbourhood basis at $0$ for a group topology.

\noindent \textbf{Convergent sequences in $(\Z,\tau_2)$}

\noindent We try to find a characterization of sequences that converge to $0$ in $\tau_2$. Choose $l_j\rightarrow 0$ in $\tau_2$; by definition of convergence, for each neighbourhood $U\in\mathcal{U}(0)$ there exists $N_0\in\N$ such that $l_j\in U$ for all $j\geq N_0$. Now, $l_j\in U_{n_0}$ if $l_j=2^{n_0}k$ where $k\in\Z$. Hence, for each $n\in\N$ there must exist $j_n$ such that for every $j\geq j_n,\ l_j\in 2^n\Z$, or, equivalently $2^n\mid l_j$. Then, we have found that:

\centerline{$l_j\rightarrow$ $0$ in $\tau_2\Longleftrightarrow \forall n\in\N\ \exists j_n\mbox{ such that for all } j\geq j_n,\ 2^n\mid l_j$}

\noindent \textbf{The dual of $(\Z,\tau_2)$}.

\noindent Let $\chi\in (\Z,\tau_2)^\wedge$.Since $\chi$ is continuous, there exists a neighbourhood $U=2^n\Z$ such that $\chi(U)=\chi(2^n\Z)\subset \T_+$. On the other hand $2^n\Z$ is a subgroup of $\Z$, hence, its image by a homomorphism will be again a subgroup in $\T$. The only subgroup contained in $\T_+$ is $\{0\}$, therefore, $\chi(2^n\Z)=\{0+\Z\}$. Let $x+\Z=\chi(1)$. Hence, $\chi(2^n)=2^nx+\Z=0+\Z$.
This implies that $2^nx\in \Z$.

\noindent We remember that $x\in\R$, and $x=\frac{k}{2^n}+\Z\in\T$. This $x$ represents the character $\chi$ we have chosen at the beginnning. Hence:

\centerline{$(\Z,\tau_2)^\wedge \subset\{\frac{k}{2^n}+\Z\mid\  k\in\Z, n\in\Z\}.$}

\noindent  Conversely, $(\Z,\tau_2)^\wedge\supset\{\frac{k}{2^n}+\Z\mid\  k\in\Z, n\in\Z\}$:

\noindent Let $k\in\Z, n\in\N$; we want to prove that $\frac{k}{2^n}+\Z\in(\Z,\tau_2)^\wedge$.

\noindent $\ker(\frac{k}{2^n}+\Z)=\{j\in\Z\mid j\frac{k}{2^n}+\Z=0+\Z\}\supset 2^n\Z$. Since any subgroup containing a neighbourhood is open, in particular, $\ker(\frac{k}{2^n}+\Z)$ is open.

\noindent Furthermore, since $(\frac{k}{2^n}+\Z)^{-1}(\T_n)\supset \ker(\frac{k}{2^n}+\Z)$ our homomorphism is continuous. Hence we have the equality between both groups.

\noindent \textbf{This group is called Pr\"ufer's group and it is denoted by $\Z(2^\infty)$.}

\section{Definition of $S$-topologies}

\noindent Let $G$ be a topological group. Our aim is to define topologies on $G$.
\noindent We choose $S\subset \Hom (G,\T)$ and set $\mathcal{U}= \{\displaystyle\bigcap_{\chi\in S}\chi^{-1}(\T_{n})\mid n\in\N\}$. It would be nice that these neighbourhoods were a neighbourhood basis at $0$. We should verify conditions in \ref{ctg}; we will denote $U_{n}=\displaystyle\bigcap_{\chi\in S}\chi^{-1}(\T_{n})$:

(i) $0\in U_{n}, \forall n\in\N$.

\noindent Trivial, since $\chi (0)=0+\Z, \forall\chi\in \Hom(G,\T)$.

(ii) $\forall U\in\mathcal{U}, \exists V\in \mathcal{U}$ such that $-V\subset U$.

\noindent We will see that $U_{n}$ is symmetric for every $n$. 
$\T_{n}$ is symmetric for every $n$. We also know that the inverse image by a continuous function of a symmetric subset is again a symmetric subset; hence $\chi^{-1}(\T_n)$ is symmetric. Since intersection of symmetric sets is again symmetric, we can assure that $U_n$ is symmetric.Thus, if $U\supset U_n$, then $V=U_n$ suits our purposes.

(iii) $\forall U\in\mathcal{U}, \exists V\in\mathcal{U}$ such that $V+V\subset U$.

\noindent Fix $U=U_{n}$. We will see that we can choose $V=U_{2n}$. Let $x, y \in U_{2n}=\{ z\mid \chi (z)\in \T_{2n}, \forall \chi \in S\}$ and let $\chi\in S$. We compute $\chi(x+y)=\chi(x)+\chi(y)\in\T_{2n}+\T_{2n}=\T_{n}$. Then $x+y\in U_{n}$, as desired.

(iv) $\forall U, V\in \mathcal{U}, \exists W \in\mathcal{U}$ such that $W\subset U\cap V$.

\noindent Let $U=U_{n}$ and $V=U_{m}$ where $n, m\in\N$. Let $k=\max(n,m)$ and $W=U_{k}$. By the definition of the neighbourhoods we have $W=U\cap V$.

\noindent By \ref{ctg}, there exists a unique group topology $\tau_{S}$
for which $\mathcal{U}$ is a neighbourhood basis at $0$.

\noindent Now, we will see that $S\subset U_{1}\tr=(\displaystyle\bigcap_{\chi\in S}\chi^{-1}(\T_{+}))\tr$. Choose $\varphi\in S$ and $x\in\displaystyle\bigcap_{\chi\in S}\chi^{-1}(\T_{+})\subset\varphi^{-1}(\T_{+})$. Hence, $\varphi(x)\in\T_{+}$.

\section{Introduction to $S$-topologies in $\Z$}

\noindent Since we are interested in topologies on $\Z$ , we should choose $S\subset \T=\Hom(\Z,\T)$.
Here $U_n=\displaystyle\bigcap_{\chi\in S}\chi^{-1}(\T_n)=\{k\in\Z$ such that  $\forall z+\Z\in S,
z\cdot k+\Z\in \T_n\}$. Related to these questions we can find this theorem in
\cite{unpublished}:

\begin{theorem}\label{tau_Sesdiscreta}
Let $\{x_n\}$ be a strictly decreasing sequence in $(0,\frac{1}{2}]$, where $x_n\rightarrow 0$,  such that $(\frac{x_n}{x_{n+1}})_{n\in\N}$ is bounded. Let $S=\{x_n+\Z\ |\ n\in\N \}$. Then $\tau_S$ is discrete.
\end{theorem}

\begin{proof}
Since $(\frac{x_n}{x_{n+1}})$ is bounded, there exists $m\in\N$, with $m>1$ such that
$\frac{x_n}{x_{n+1}}\leq m,\ \forall n \in\N$. Choose $k\in\Z$ where $k\in\displaystyle\bigcap_{\chi\in S}\chi^{-1}(\T_m)$. Since $\T_m$ is symmetric we can pick $k\geq 0$.

\noindent First, we prove $k\leq\frac{1}{4x_1}$. By contradiction, suppose that $k>\frac{1}{4x_1}$. Since $(x_n)_{n\in\N}$ is strictly decreasing, there exists a unique $n\in\N$ such that $\frac{1}{4x_n}<k\leq \frac{1}{4x_{n+1}}$. Multiplying by $x_{n+1}$ we obtain that $\frac{1}{4}\frac{x_{n+1}}{x_n}<kx_{n+1}\leq\frac{1}{4}$. Since $m$ is a bound for $\frac{x_n}{x_{n+1}}$ we get $\frac{1}{4m}<kx_{n+1}\leq\frac{1}{4}$. Hence $kx_{n+1}+\Z\notin \T_m$. Choosing $\chi:k\mapsto k\cdot x_{n+1}+\Z$ we get that $\chi(k)\notin\T_m$. Which contradicts $k\in\displaystyle\bigcap_{\chi\in S}\chi^{-1}(\T_m)$.

\noindent Hence, $\displaystyle\bigcap_{\chi\in S}\chi^{-1}(\T_m)\subset\Z\cap[-\frac{1}{4x_1},\frac{1}{4x_1}]$,
 which is a finite subset. Now, fix $l\in\N$ such that
$4lx_1>1$. Considering $j\in\displaystyle\bigcap_{\chi\in S}\chi^{-1}(\T_{lm } )$ we get $j,2j,\dots ,lj\in \displaystyle\bigcap_{\chi\in S}\chi^{-1}(\T_m)\subset\displaystyle[-\frac{1}{4},\frac{1}{4}]\cap\Z $.
Equivalently:  $j\in[-\frac{1}{4lx_1},\frac{1}{4lx_1}]\cap\Z$, by the choice of $l$ this means that $j=0$. $\displaystyle\bigcap_{\chi\in S}\chi^{-1}(\T_{lm})=\{0\}$. Hence $\tau_S$ is discrete.

\end{proof}

\noindent However, this proposition does not give any information if $(\frac{x_n}{x_{n+1}})$ is an unbounded sequence. In the following pages, we will study some $S$-topologies where $S=\{x_n\mid n\in\N\}$ with $(\frac{x_n}{x_{n+1}})$ unbounded.

\section{The $S$-topology corresponding to $S=\{2^{-n^2}+\Z\mid n\in\N\}\subset\T$}

\noindent  We remember that in order to get a $S$-topology in $\Z$, we should choose $S\subset \Hom(\Z,\T)\approx \T$. We will consider, in particular sequences $(a_n)_{n\in\N}\subset\N$ strictly increasing and we choose $S=\{2^{-a_n}+\Z\mid n\in\N\}$. We have already seen that if $\mid\frac{2^{-a_n}}{2^{-a_{n+1}}}\mid= \mid\frac{2^{a_{n+1}}}{2^{a_{n}}}\mid= 2^{a_{n+1}-a_n}$ is bounded, then $\tau_S$ is discrete. 

\noindent Then, $\tau_S$ has then as neighbourhood basis at $0:\ V_{S,n}=\{k\in\Z\mid\forall z+\Z\in S, \ z\cdot k+\Z\in 	\T_n\}$.

\noindent Our objective will be to find $(\Z,\tau_S)^\wedge$ where $(a_{n+1}-a_n)_{n\in \N}$ is an unbounded sequence. For example, we could consider $S_1=\{2^{-2^n}+\Z\mid n\in\N\}$, $S_2=\{2^{-n^2}+\Z\mid n\in\N\}$, $S_3=\{2^{-n!}+\Z\mid n\in\N\}$. We will focus our attention in the case $S=\{2^{-n^2}+\Z\mid n\in\N\}$.

\noindent In order to study the topology, we will study its convergent sequences. It will also be useful to find $(\Z,\tau_S)^\wedge$, since $\chi\in(\Z,\tau_S)^\wedge \Longleftrightarrow \forall (k_n)\subset\Z$ where $k_n\rightarrow 0$ in $\tau_S$, then $\chi(k_n)\rightarrow 0+\Z$.

\noindent Fix now $S=\{2^{-n^2}+\Z\mid n\in\N\}$. We want to find sequences converging to zero in $\tau_S$. For that purpose, we will find a characterization of those sequences.

\noindent $l_j\rightarrow 0$ in $\tau_S \Longleftrightarrow \forall m\in \N, \exists j_m$ such that $l_j\in V_{S,m}$ for all $j\geq j_m$. What does this mean? 
$l_j\in V_{S,m}=\{k\in \Z\mid \forall x+\Z\in S,\ kx+\Z\in\T_m\}= \{k\in \Z\mid \frac{1}{2^{n^2}}k+\Z\in\T_m, \forall n\in\N\}$.
This means, $\frac{l_j}{2^{n^2}}+\Z\in\T_m \ \forall n\in\N$. Hence

\centerline{$l_j\rightarrow 0$ in $\tau_S \Longleftrightarrow \forall m\in\N \ \exists j_m,$ such that $\forall n\in\N , \frac{l_j}{2^{n^2}}+\Z\in\T_m$ for all $j\geq j_m$.}

\noindent Now, as we have different criteria for convergence in $\tau_S$ and $\tau_2$ we will see the realtionship between them.

\begin{proposition}
$l_j\rightarrow 0$ in $\tau_S \Rightarrow l_j\rightarrow 0$ in $\tau_2$. Consequently: $\tau_2\leq\tau_S$
\end{proposition}

\begin{proof}
Let $n_0\in\N$. We will choose $m=2^{n_0^2}$. By hypothesis, there exists $j_m$ such that $\frac{l_j}{2^{n^2}}+\Z\in\T_m, \forall n\in\N,$ if $j\geq j_m$. Equivalently, $\frac{l_j}{2^{n^2}}+\Z,\frac{2l_j}{2^{n^2}}+\Z,\dots \frac{ml_j}{2^{n^2}}+\Z\in\T_+$. Which means,  $\frac{l_j}{2^{n^2}}+\Z,\frac{2l_j}{2^{n^2}}+\Z,\dots , \frac{2^{n_0^2}l_j}{2^{n^2}}+\Z\in\T_+; \forall n\in\N$. In particular, this is true for $n=n_0$, then $\frac{l_j}{2^{n_0^2}}+\Z,\frac{2l_j}{2^{n_0^2}}+\Z,\dots , \frac{ml_j}{2^{n_0^2}}+\Z\in\T_+$. But, $\{\frac{l_j}{2^{n_0^2}}+\Z,\frac{2l_j}{2^{n_0^2}}+\Z,\dots , \frac{ml_j}{2^{n_0^2}}+\Z\}=\langle\frac{l_j}{2^{n_0^2}}+\Z\rangle$. Since it is a subgroup contained in $\T_+$, it must be $\frac{l_j}{2^{n_0^2}}+\Z=0+\Z$. Therefore, $\frac{l_j}{2^{n_0^2}}\in\Z$; or, equivalently $2^{n_0^2}|l_j$. This was our convergence criterion in $\tau_2$.

\end{proof}

\noindent By this proposition, we get that if a sequence does not converge in $\tau_2$ it will neither converge in $\tau_S$, which is a useful fact, because the convergence criterion in $\tau_2$ is much easier than the one in $\tau_S$.

\noindent From here it is also derived the continuity of

\centerline {$\id:(\Z,\tau_S)\rightarrow (\Z,\tau_2)$.}

\noindent We now focus on homomorphisms between dual groups induced by continuous homomorphisms between the original groups.

\begin{deff}
Let $\varphi:G\rightarrow H$ be a continuous homomorphism between Hausdorff topological groups. Define $\varphi^\wedge : H^\wedge \rightarrow G^\wedge$ as $\chi \mapsto \chi \circ \varphi$. It is called the dual homomorphism.
\end{deff}

\begin{proposition}
Let $\varphi: G\rightarrow H$ be a continuous homomorphism. Then $\varphi^\wedge$ is continuous
\end{proposition}

\begin{proof}
As usual, we choose a neighbourhood of $G^\wedge$, compute its inverse image and see if it is a neighbourhood of $H^\wedge$. Let $K\subset G$ be a compact subset, $K\tr\subset\mathcal{U}_{G^\wedge}(0) $.

\noindent In order to see that $(\varphi^\wedge)^{-1}(K\tr)\in \mathcal{U}_{H^\wedge}(0)$, we compute $\varphi(K)\tr$ and try to verify $\varphi(K)\tr\subset (\varphi^\wedge)^{-1}(K\tr)$. 

\noindent $\varphi(K)$ is compact and, hence $\varphi(K)\tr$ is a neighbourhood of $0$ in $H^\wedge$

\noindent Let $\psi\in\varphi(K)\tr \Longleftrightarrow \psi(\varphi(K))\subset\T_+$.

\noindent We need that $\varphi^\wedge(\psi)\in K\tr$. We fix $x\in K$ and compute $\varphi^\wedge(\psi(x))=\psi(\varphi(x))\in \psi(\varphi(K))\subset \T_+$. And the result follows.

\end{proof}

\noindent In our case, the identity function $\id:(\Z,\tau_S)\rightarrow (\Z,\tau_2)$ induces a dual homomorphism:

\centerline {$\id^\wedge:(\Z,\tau_2)^\wedge =\Z(2^\infty)\rightarrow (\Z,\tau_S)^\wedge$}

\noindent What else can we say about $\id^\wedge$? We will prove an easy but useful result:

\begin{proposition}
If $\varphi:G\rightarrow H$ is onto, then $\varphi^\wedge:H^\wedge\rightarrow G^\wedge$ is one-to-one.
\end{proposition}

\begin{proof}
Take $\chi\in\ker\varphi_{G^\wedge}$. Observe that: $0_{G^\wedge}=\varphi^\wedge(\chi)=\chi o\varphi \Longleftrightarrow\ \forall x\in G: \chi(\varphi(x))=0+\Z$.

\noindent Since $\varphi$ is onto, we obtain $\forall h\in H, \chi(h)=0+\Z$. That means, $\chi=0_{H^\wedge}$.

\noindent Hence $\varphi^\wedge$ is one-to-one.

\end{proof}

\noindent Since $\id:(\Z,\tau_S)\rightarrow (\Z,\tau_2)$ is onto, $\id^\wedge:(\Z,\tau_2)^\wedge \rightarrow (\Z,\tau_S)^\wedge$ is one-to-one. Hence, we have proved, for our $S$, that the Pr\"ufer's group $\Z(2^\infty)$ can be embedded in $(\Z,\tau_S)^\wedge$. For any other $S$ we would only need to prove the continuity of the corresponding identity function, as we will see in next section.

\noindent With this new information we keep on seeking sequences converging to $0$ in $\tau_S$.

\noindent We will take advantage on $l_j\rightarrow 0$ in $\tau_S \Rightarrow l_j\rightarrow 0$ in $\tau_2$. Hence, $l_j\rightarrow 0$ implies that $\forall n\in\N\ \exists j_n,\ 2^{n^2}|l_j$ for all $j\geq j_n$. Choose $j_n$ minimal, and such that $2^{n^2}|l_j \ j\geq j_n$. We get an increasing sequence $j_1\leq j_2\leq j_3\cdots$

\noindent Define:

\noindent $M_0:=\{1,2,\cdots,j_1-1\}$.

\[
M_n=\left\{
\begin{array}{c@{\quad:\quad}c}
\{j_n\}& j_n=j_{n+1} \\
\{j_n,j_n+1,\cdots j_{n+1}-1\} & j_n<j_{n+1}%
\end{array}%
\right.
\]

\noindent  $S_n=\displaystyle\max_{j\in M_n}\frac{|l_j|}{2^{(n+1)^2}}\in \R$

\begin{proposition}
$S_n\rightarrow 0$ in $\R \Rightarrow l_j\rightarrow 0$ in $\tau_S$.
\end{proposition}

\noindent A similar result can be found in page 115 of \cite{tesislorenzo}, but it is not correct; in fact we have a counterexample below.

\begin{proof}
Let $m\in\N$. By hypothesis, $\exists n_0 \mbox{ such that } S_n<\frac{1}{4m}$ where $n\geq n_0$.
Let $j \in M_n$ for some $n\ge n_0$. We want to see if $l_j\in V_{S,m}$. For that, it must be $\frac{l_j}{2^{k^2}}+\Z\in\T_m$.

\noindent If $k\leq n$ then $2^{k^2}|2^{n^2}|l_j \Longrightarrow \frac{l_j}{2^{k^2}}+\Z=0+\Z\in\T_m$.

\noindent If $k\geq n+1$ we have that $0\leq \frac{|l_j|}{2^{k^2}}\leq \frac{|l_j|}{2^{(n+1)^2}}\leq S_n<\frac{1}{4m}\Longrightarrow \frac{l_j}{2^{(n+1)^2}}\in \T_m$

\end{proof}

\noindent Although the result mencioned above is not true, we can prove the following result:

\begin{proposition}
Let:

(i) $S_n\rightarrow 0$ in $\R$

(ii) $l_j\rightarrow 0$ in $\tau_S$

(iii) $S_n+\Z \rightarrow 0+\Z$ in $\T$

\noindent Then (i) $\Longrightarrow$ (ii) and (ii)  $\Longrightarrow$ (iii)
\end{proposition}

\noindent Obviously, (iii)$\nRightarrow$(i), hence the three conditions are not equivalent, now we are interested in what happens with (ii)$\Rightarrow$(i). In fact we have two examples in which we can see that (ii)$\nRightarrow$(i).

\noindent The ideas and examples up to the end of the section were obtained in collaboration with Lydia Aussenhofer during her stay at the UCM.

\noindent We are seeking examples of (ii) $\nRightarrow$ (i). For that we will construct sequences in the form $l_j=2^{j^2}\cdot a_j$ where $a_j$ will be an odd number for every $j$. In this way we will get that $M_n=\{n\};\ j_n=\{n\};\ y\ S_n=\frac{l_n}{2^{(n+1)^2}}$.

\begin{example} $(S_n)_{n\in\N}$ bounded.

\noindent  We want that $S_n \rightarrow 1$. We know that $S_n=\frac{l_n}{2^{(n+1)^2}}=\frac{2^{n^2}\cdot a_n}{2^{n^2+2n+1}}= \frac{a_n}{2^{2n+1}}\approx 1$. So, we fix $a_n=2^{2n+1}-1$ (remember that it must be an odd number), so $l_j=2^{(j+1)^2}-2^j$.

\noindent $S_n+\Z=\frac{2^{2n+1}-1}{2^{2n+1}}+\Z= 1-\frac{1}{2^{2n+1}}+\Z\rightarrow 0+\Z$ in $\T$. (With this example we do not know if (iii)$\Longrightarrow $(ii) is false or not). But $S_n\nrightarrow 0$ in $\R$. (with this we will get (ii) $\nRightarrow$ (i))

\noindent We will see that, in fact, $l_j\rightarrow 0$ in $\tau_S$: Let $m\in\N$. For $j_0=m$ and $j\ge j_0$ we have $2^{-2j-1}+\Z\in\T_m$.

\noindent As before, for $n\leq j$ the proof is completely trivial.

\noindent We have in fact $n=j+k; k\geq 1 $

\noindent $\frac{l_j}{2^{(j+k)^2}}+\Z\in \T_m$?

\noindent $\frac{l_j}{2^{(k+j)^2}}=\frac{2^{(j+1)^2}}{2^{(j+k)^2}}-\frac{2^{j^2}}{2^{(j+k)^2}}=
\frac{2^{j^2+2j+1}}{2^{j^2+2jk+k^2}}-\frac{2^{j^2}}{2^{j^2+2jk+k^2}}= 2^{-2j(k-1)-k^2+1}-2^{-2jk-k^2}$

\noindent If $k=1,\ 2^{-2j(k-1)-k^2+1}-2^{-2jk-k^2}=2^0-2^{-2j-1}$. But, $2^0-2^{-2j-1}+\Z=-2^{-2j-1}+\Z\in\T_m$.

\noindent If $k\geq 2$
\[
\left\{
\begin{array}{c}
-2j(k-1)-k^2+1<<-2j \Longrightarrow  2^{-2j(k-1)-k^2+1}+\Z\in\T_{2m}\\
y\\
-2jk-k^2<<-2j \Longrightarrow -2^{-2jk-k^2}+\Z\in\T_{2m}%
\end{array}%
\right.
\]

\noindent Hence, $2^{-2j(k-1)-k^2+1}-2^{-2jk-k^2}+\Z\in\T_{2m}+\T_{2m}=\T_m$.

\end{example}

\begin{example} $(S_n)_{n\in\N}$ unbounded.

\noindent We will fix $S_n$ increasing as n; $S_n=\frac{2^{n^2}\cdot a_n}{2^{(n+1)^2}}=\frac{a_n}{2^{2n+1}}\approx n \Longrightarrow a_n\approx n\cdot 2^{2n+1}$.

\noindent As we want that $a_n$ be an odd number we get $a_n=n2^{2n+1}-1$. Hence $l_j=2^{j^2}(j2^{2j+1}-1)=j2^{(j+1)^2}-2^{j^2}$.

\noindent Obviously $S_n\nrightarrow 0$ in $\R$.

\noindent We will see $l_j\rightarrow 0$ in $\tau_S$. As before, we want that $\frac{l_j}{2^{n^2}}+\Z\in\T_m$; and, also as before, the case $n\leq j$ does not matter. Now, we consider $m_1$ such that $m_12^{-2m_1}<\frac{1}{8m}$. Let $j_0=\max(m,m_1)$ and $j\ge j_0$.

\noindent $\frac{l_j}{2^{(k+j)^2}}=\frac{j2^{(j+1)^2}}{2^{(j+k)^2}}- \frac{2^{j^2}}{2^{(j+k)^2}}= j2^{-2j(k-1)-k^2+1}-2^{-2jk-k^2}$

\noindent If $k=1$ then $j2^{-2j(k-1)-k^2+1}-2^{-2jk-k^2}=j-2^{-2j-1}$ and $j-2^{-2j-1}+\Z=-2^{-2j-1}+\Z\in\T_m$.

\noindent If $k>1$,

\noindent $-2j(k-1)-k^2+1<<-2j$; hence $j2^{-2j(k-1)-k^2+1}\leq j2^{-2j} \leq \frac{1}{8m}$.  Thus, $j2^{-2j(k-1)-k^2+1}+\Z\in\T_{2m}$.

\noindent $-2jk-k^2<<-2j$. Then, $-2^{-2jk-k^2}+\Z\in\T_{2m}$.

\noindent $j2^{-2j(k-1)-k^2+1}-2^{-2jk-k^2}+\Z\in\T_{2m}+\T_{2m}=\T_m$.

\noindent Thus, $l_j\rightarrow 0$ in $\tau_S$.

\end{example}

\noindent The following example shows that (iii)$\nRightarrow$(ii).

\begin{example}

\noindent Define: 

\[l_j=\left\{
\begin{array}{c@{\quad:\quad}c}
2^{(n+1)^2}& j=n^2-2 \\
2^j & otherwise%
\end{array}%
\right.
\]

\noindent $j_1=1$, $j_2=4$, $j_3=9$, $\dots $, $j_n=n^2$.

\noindent $M_1=\{1,2,3\}$, $M_2=\{4,5,6,7,8\}$, $M_3=\{9,\dots , 15\}$, $\cdots$, $M_n=\{n^2,\dots ,(n+1)^2-1\}$.

\noindent $S_1=\frac{l_2}{2^{2^2}}=1$, $S_2=\frac{l_7}{2^{2^3}}=1$, $S_n=\frac{l_{n^2-2}}{2^{(n+1)^2}}=1$. Hence $S_n+\Z\rightarrow 0+\Z$.

\noindent It remains to see that $l_j\nrightarrow 0$ in $\tau_S$.

\noindent We observed that, $l_j\rightarrow 0$ in $\tau_S\Longleftrightarrow \forall m\in\N\, \exists j_m$ such that $\forall n$ $\frac{l_j}{2^{n^2}}+\Z\notin\T_m$ for all $j\geq j_m$.

\noindent Let $m=1$. It suffices to show that $A=\{j_m\in\N\mid \exists n$ such that $\frac{l_j}{2^{n^2}}+\Z\in\T_+$ for all $j\geq j_m\}$ is cofinal in $\N$.

\noindent Fix $n$. Find $j_n$ such that $\frac{l_{j_n}}{2^{n^2}}=\frac{1}{2}\Longleftrightarrow l_{j_n}=2^{n^2-1}\Longleftrightarrow j_n=n^2-1$.

\noindent Thus, $A=\{n^2-1\mid n>1\}$ and it is obviously cofinal.

\noindent Hence $l_j\nrightarrow 0$ in $\tau_S$

\end{example}

\section{General case $S=\{2^{-a_n}+\Z\ |\ n\in\N\}$}

\noindent We must impose certain restrictions to the sequences $(a_n)$. For example, $a_n\in\N, \ \forall n$, furthermore, we want $(a_n)$ to be a strictly increasing sequence, where $a_n\rightarrow \infty$ and $(a_{n+1}-a_n)_{n\in\N}$ be an unbounded sequence.

\noindent In this case: $V_{S,m}=\{k\in\Z|\ \forall x+\Z\in S,\ x\cdot k+\Z\in\T_m\}$.

\noindent We seek a convergence criterion, for our new $\tau_S$:
 $l_j\rightarrow 0$ in $\tau_S$
$\Longleftrightarrow \ \forall m\in\N \ \exists j_m \mbox{ such that }l_j\in V_{S,m}$ for all  $j\geq j_m\ \Longleftrightarrow\
\forall m\in\N,\ \exists j_m\mbox{ such that } \frac{l_j}{2^{a_n}}+\Z\in\T_m$ if $j\geq j_m,\ \forall n\in\N$

\noindent We want to see the sequencial continuity of $\id :(\Z,\tau_S)\rightarrow (\Z,\tau_2)$. Since $\tau_S$ is a metrizable topology, continuity and sequencial continuity are equivalent.
For that, we must prove the following result:

\begin{proposition}
$l_j\rightarrow 0$ in $\tau_S$ $\Rightarrow$ $l_j\rightarrow 0$ in $\tau_2$.
\end{proposition}

\noindent For this result it is only necessary that $(a_n)_{n\in\N}$ is a sequence of natural numbers. As we want $(a_n)$ to be an increasing sequence, we can also include this hypothesis.

\begin{proof}
Let $n_0\in\N$. We fix $m=2^{a_{n_0}}$. By hypothesis there exists $j_m$ such that $\frac{l_j}{2^{a_n}}+\Z\in \T_m\ \forall n$ for all $j\geq j_m$. This means, $\frac{l_j}{2^{a_n}}+\Z, \frac{2l_j}{2^{a_n}}+\Z,\dots, \frac{ml_j}{2^{a_n}}+\Z\in\T_+\ \forall n$ $\Longleftrightarrow$ $\frac{l_j}{2^{a_n}}+\Z, \frac{2l_j}{2^{a_n}}+\Z,\dots, \frac{2^{a_{n_0}}l_j}{2^{a_n}}+\Z\in\T_+\ \forall n$. In particular for $n=n_0$, we have $\frac{l_j}{2^{a_{n_0}}}+\Z, \frac{2l_j}{2^{a_{n_0}}}+\Z,\dots, \frac{2^{a_{n_0}}l_j}{2^{a_{n_0}}}+\Z\in\T_+$. But $\frac{l_j}{2^{a_{n_0}}}+\Z, \frac{2l_j}{2^{a_{n_0}}}+\Z,\dots, \frac{2^{a_{n_0}}l_j}{2^{a_{n_0}}}+\Z=\langle \frac{l_j}{2^{a_{n_0}}}+\Z \rangle\subset\T_+$. Thus, $\frac{l_j}{2^{a_{n_0}}}+\Z=0+\Z$. This means, $\frac{l_j}{2^{a_{n_0}}}\in\Z$; or, $2^{a_{n_0}}\mid l_j$. Hence, $l_j\rightarrow 0$ in $\tau_2$.

\end{proof}

\begin{remark}\label{3.5.2}
The previous proposition implies that 

\centerline{$\id:(\Z,\tau_S)\rightarrow(\Z,\tau_2)$}

\noindent is continuous. Since it is also surjective, the mapping:

\centerline{$\id^{\wedge}:(\Z,\tau_2)^\wedge\rightarrow(\Z,\tau_S)^\wedge$} 

\noindent is injective, and $(\Z,\tau_2)^\wedge\subseteq(\Z,\tau_S)^\wedge$

\end{remark}

\noindent In the same line of the previous results for  $S=\{2^{-n^2}\mid n\in\N\}$, we define $j_n$ minimal, such that $2^{a_n}\mid l_j$ for all $j\geq j_n$. Define $M_n$ as in the previous section and $S_n=\max\{\frac{\mid l_j\mid}{2^{a_{n+1}}}|j\in M_n\}$

\begin{proposition}
$S_n\rightarrow 0$ in $\R$ $\Rightarrow$ $l_j\rightarrow 0$ in $\tau_S$.
\end{proposition}

\begin{proof}
Let $m\in \N$. Since $S_n\rightarrow 0$ in $\R$, $\exists n_0$ such that $S_n<\frac{1}{4m}$ where $n\geq n_0$. Choose $j\in M_n$ where $n>n_0$.

\noindent We want to see that $l_j\in V_{S,m}$. For that, $\frac{l_j}{2^{a_k}}+\Z\in\T_m,\ \forall k$.

\noindent If $k\leq n$, then $a_k\leq a_n$; and hence, $2^{a_k}\mid 2^{a_n}\mid l_j$. That is, $\frac{l_j}{2^{a_k}}+\Z=0+\Z\in\T_m$.

\noindent If $k\geq n+1$, we have $0\leq\frac{\mid l_j\mid}{2^{a_k}}\leq\frac{\mid l_j\mid}{2^{a_{n+1}}}\leq S_n<\frac{1}{4m}$. Hence, $\frac{l_j}{2^{a_k}}\in\T_m$.

\end{proof}

\noindent We will see that $l_j\rightarrow 0$ in $\tau_S$ $\nRightarrow$ $S_n\rightarrow 0$ in $\R$. Choose $l_j=2^{a_j}b_j$ where $b_j$ is an odd number. Thus, $j_n=\{n\},\ M_n=\{n\}$ and $S_n=\frac{l_n}{2^{a_{n+1}}}$.

\begin{example} $S_n\rightarrow 1$

\noindent Now we will suppose that $(a_{n+1}-a_n)_{n\in\N}$ is a strictly increasing sequence. If we delete this hypothesis, the example is not valid as counter example of the implication we are studying, remainig as an open question.

\noindent $S_n=\frac{2^{a_n}b_n}{2^{a_{n+1}}}\approx 1$. Then, $b_n\approx 2^{a_{n+1}-a_n}$. Now, since $b_n$ must be odd, we choose $b_n=2^{a_{n+1}-a_n}-1$. Thus, $l_j=2^{a_{j+1}}-2^{a_j}$.

\noindent We will see now that $l_j\rightarrow 0$ in $\tau_S$. For that, we fix $m$, and choose $j_0=2m$. Let $j\geq j_0$

\noindent We must see that $\frac{l_j}{2^{a_{n}}}+\Z\in\T_m$. That means, $\frac{2^{a_{j+1}}-2^{a_j}}{2^{a_n}}+\Z\in\T_m \ \forall n$.

\noindent If $n\leq j$, then $2^{a_n}\mid 2^{a_{j+1}}$ and $2^{a_n}\mid 2^{a_{j}}$, hence $\frac{2^{a_{j+1}}-2^{a_j}}{2^{a_n}}+\Z=0+\Z\in\T_m$.

\noindent We fix now $n=j+k$. We must see $\frac{2^{a_{j+1}}}{2^{a_{j+k}}}-\frac{2^{a_{j}}}{2^{a_{j+k}}}+\Z\in\T_m$.

\noindent If k=1 $\frac{2^{a_{j+1}}}{2^{a_{j+k}}}-\frac{2^{a_{j}}}{2^{a_{j+k}}}+\Z=-\frac{2^{a_j}}{2^{a_{j+1}}}+\Z$. By the choice of $j_0$ we have $-\frac{2^{a_j}}{2^{a_{j+1}}}+\Z\in\T_m$.

\noindent If $k>1$, we must see $2^{a_{j+1}-a_{j+k}}-2^{a_j-a_{j+k}}+\Z\in\T_m$. By the choice of $j_0$ we have that $2^{a_{j+1}-a_{j+k}}\in\T_{2m}$ and $2^{a_j-a_{j+k}}\in\T_{2m}$. Thus, $2^{a_{j+1}-a_{j+k}}-2^{a_j-a_{j+k}}+\Z\in\T_m$.

\end{example}

\begin{example}$S_n$ unbounded

\noindent As in the previous example we will need $(a_{n+1}-a_n)_{n\in\N}$ be a strictly increasing sequence.

\noindent Fix $S_n=\frac{2^{a_n}b_n}{2^{a_{n+1}}}\approx n$. Choose $b_n=n2^{a_{n+1}-a_{n}}-1$. Thus, $l_j=j2^{a_{j+1}}-2^{a_j}$.

\noindent We want to see that $l_j\rightarrow 0$ in $\tau_S$; which means, $\frac{l_j}{2^{a_n}}+\Z\in\T_m\ \forall n$.

\noindent We fix $m\in \N$. Let $m_1$ such that $m_12^{a_{m_1+1}-a_{m_1}}\leq \frac{1}{8m} $. Let $j_0=\max\{m_1,2m\}$. Let $j\geq j_0$.

\noindent If, $n\leq j$, as before, the result follows trivially.

\noindent We choose, then, $\frac{j2^{a_{j+1}}-2^{a_j}}{2^{a_{j+k}}}+\Z=j2^{a_{j+1}-a_{j+k}}-2^{a_j-a_{j+k}}+\Z$.

\noindent If $k=1$, the expression turns into $j-2^{a_{j+1}-a_j}+\Z\in\T_m$, by the choice of $j_0$.

\noindent If $k>1$:

\noindent By the choice of $j_0$ we have $j2^{a_{j+1}-a_{j+k}}\leq\frac{1}{8m}$. That is, $j2^{a_{j+1}-a_{j+k}}+\Z\in\T_{2m}$.

\noindent Also, $-2^{a_j-a_{j+k}}+\Z\in\T_{2m}$.

\noindent Hence, $j2^{a_{j+1}-a_{j+k}}-2^{a_j-a_{j+k}}+\Z\in\T_m$.

\end{example}

\chapter{An approach to the dual of $(\Z,\tau_S)$ for $S=\{2^{-a_n}+\Z\mid n\in\N\}$}
\section{Writing of a natural number as series}

\noindent We produce a device in order to express a natural number as a sum pivoted by a particular sequence of natural numbers.

\begin{proposition}\label{bn}
Let $(b_n)_{n\in\N_0}$ a sequence of natural numbers such that $b_0=1$, $b_n\neq b_{n+1}$ and $b_n\mid b_{n+1},$ for all $n \in \N$. Then, for each natural number $l\in\N$, there exists a natural number $N(l)$, and integers
$k_0,\dots ,k_{N(l)}$, such that $l=\sum_{i=0}^{N(l)}k_ib_i$ and such that
$\mid k_n\mid\leq \frac{b_{n+1}}{2b_n}$, for $0\leq n\leq N(l)$. Also, $\mid\sum_{i=0}^{n}k_ib_i\mid\leq\frac{b_{n+1}}{2}$ for all $n$.
\end{proposition}

\begin{proof}
\noindent Consider $rd(r)$ as the closest integer to r. In case $r=k+0.5$, where $k\in\Z$ 
choose $rd(r)$ as the closest integer to $0$ between $k$ and $k+1$. It is obvious that $|rd(x)|=rd(|x|)$

\noindent Fix $l\in\Z$.

\noindent Let $N$ be the minimum natural number such that $b_N\geq \mid l\mid$. 

\noindent Put $k_n=0$ for $n>N$. Let us define recursively the coefficients $k_N,\ldots, k_1,k_0$ 
in the following way: define $k_N:=rd(\frac{l}{b_N})$. Once defined $k_N,\dots,k_{n+1}$, define 
$k_n:=rd(\frac{l-\sum_{i=n+1}^{N}k_ib_i}{b_n})$. As a consequence, we obtain 
$k_{0} =rd(\frac{l-\sum_{i=1 }^{N}k_ib_i}{b_0})=rd(  l-\sum_{i=1 }^{N}k_ib_i)=l-\sum_{i=1 }^{N}k_ib_i$ 
and hence $l=\sum_{i=0}^{N}k_ib_i$.

\noindent Now let us see that 
$\mid\sum_{i=0}^{n}k_ib_i\mid \leq \frac{b_{n+1}}{2}$ for any $l$,  $0\le n\le N$.

\noindent $\mid\sum_{i=0}^{n}k_ib_i\mid=\mid(l-\sum_{i=n+2}^{N}k_ib_i)-k_{n+1}b_{n+1}\mid
= \mid b_{n+1}(\frac{l-\sum_{i=n+2}^{N}k_ib_i}{b_{n+1}})-b_{n+1}rd(\frac{l-\sum_{i=n+2}^{N}k_ib_i}{b_{n+1}})
\mid= b_{n+1}\mid\frac{l-\sum_{i=n+2}^{\infty}k_ib_i}{b_{n+1}}-rd(\frac{l-\sum_{i=n+2}^{N}
k_ib_i}{b_{n+1}})\mid\leq \frac{b_{n+1}}{2}$.

\noindent Finally we want to show that  $\mid k_n\mid\leq \frac{b_{n+1}}{2b_n}$.\label{acotkn}

\noindent $\mid k_n\mid=
\mid rd(\frac{l-\sum_{i=n+1}^N k_ib_i}{b_n})\mid=
rd(\frac{\mid l-\sum_{i=n+1}^N k_ib_i\mid}{b_n})=
rd(\frac{\mid \sum_{i=0}^n k_ib_i\mid}{b_n})$.

\noindent Since the round function is non-decreasing and 
$\mid \sum_{i=0}^n k_ib_i\mid\leq \frac{b_{n+1}}{2}$, 
$rd(\frac{\mid \sum_{i=0}^n k_ib_i\mid}{b_n})\leq rd(\frac{b_{n+1}}{2b_n})$.

\noindent If $\frac{b_{n+1}}{b_n}$ is even, then $rd(\frac{b_{n+1}}{2b_n})=\frac{b_{n+1}}{2b_n}$. 
If $\frac{b_{n+1}}{b_n}$ is odd, then $rd(\frac{b_{n+1}}{2b_n})=\frac{b_{n+1}}{2b_n}-0.5\leq \frac{b_{n+1}}{2b_n}$.

\noindent In both cases $\mid k_n\mid\leq\frac{b_{n+1}}{2b_n}$

\end{proof}

\section{Characterization of convergent sequences in $\tau_S$. Particular case $S=\{2^{-n^2}+\Z\mid n\in \N\}$}

\noindent Let $(l_j)_{j\in\N}$ be a sequence in $\Z$ converging to $0$ in $\tau_S$. By the proposition in the previous section, we can write $l_j=\sum_{q}k_{j,q}2^{a_q}$, where $(a_n)_{n\in\N}$ a strictly increasing sequence (at some point of the chapter, we may put more conditions on $(a_n)$).

\noindent Our objective is to find necessary and/or sufficient conditions on the coefficients $k_{j,q}$ in order that the sequence $(l_j)$ converges to $0$.

\noindent We begin with our special case $S=\{2^{-n^2}+\Z\mid n\in\N\}$. We now put $l_j=\sum_q k_{j,q}2^{q^2}$.

\noindent In the previous chapter, we saw that $l_j\rightarrow 0 \Longleftrightarrow \forall m\in\N,\ \exists j_m$ such that $\forall n\in\N,\ \frac{l_j}{2^{n^2}}+\Z\in\T_m$ for all $j\geq j_m$. 

\noindent Since $l_j=\sum_q k_{j,q}2^{q^2}$, we can write $\frac{l_j}{2^{n^2}}=\sum_q k_{j,q}\frac{2^{q^2}}{2^{n^2}}$.

\noindent We observe that $\sum_q k_{j,q}\frac{2^{q^2}}{2^{n^2}}+\Z=\sum_{q=0}^{n-1} k_{j,q}\frac{2^{q^2}}{2^{n^2}}+\Z,\ \forall n\in\N$.

\noindent What do these conditions mean? 

\noindent For $n=1$, $k_{j,0}\frac{1}{2}+\Z\in\T_m$.

\noindent For $n=2$, $k_{j,0}\frac{1}{2^4}+k_{j,1}\frac{2}{2^4}+\Z\in\T_m$

\noindent For $n=3$, $k_{j,0}\frac{1}{2^9}+k_{j,1}\frac{2}{2^9}+k_{j,2}\frac{2^4}{2^9}+\Z\in\T_m$

\noindent For $n=4$, $k_{j,0}\frac{1}{2^{16}}+k_{j,1}\frac{2}{2^{16}}+k_{j,2}\frac{2^4}{2^{16}}+ k_{j,3}\frac{2^9}{2^{16}}+\Z\in\T_m$

\noindent We study the case $n=1$, $k_{j,0}\in\{0,1\}$, then $\frac{k_{j,0}}{2}+\Z\in\{0+\Z,\frac{1}{2}+\Z\}$. Since $\frac{k_{j,0}}{2}+\Z\in\T_m$, we get that $k_{j,0}=0$ for all $j>j_1$.

\noindent For $n=2$, we choose $j>j_1$, then the condition turns into $\frac{k_{j,1}}{8}+\Z\in\T_m$. For $m\geq 4$, and $j\geq j_4$,  $k_{j,1}=0$.

\noindent Repeating the same arguments for each $q$, we choose $m=(q+1)^2$; for all $j\geq j_{(q+1)^2}$, we will get that $k_{j,q}=0$.




\section{Characterization of $V_{S,m}$. Particular case $S=\{2^{-n^2}+\Z\mid n\in\N\}$}

\noindent In this section we try to find some conditions on an integer $k$ to be in $V_{S,m}$.

\noindent In this case, we will have that $k=\sum_{i=0}^{N(k)}2^{i^{2}}k_i $ and $\mid k_i\mid \leq \frac{2^{(i+1)^2}}{2\cdot 2^{i^2}}=2^{2i}$.

\noindent For our purposes we will use this technical lemma:

\begin{lemma}\label{dostercios}

$\mid \frac{k_1}{2^{(N+1)^2-1}}+\frac{k_2}{2^{(N+1)^2-4}}+\cdots + \frac{k_N}{2^{(N+1)^2-N^2}}\mid\leq \frac{2}{3}$, for every $N$.
\end{lemma}

\begin{proof}
$\mid \frac{k_1}{2^{(N+1)^2-1}}+\frac{k_2}{2^{(N+1)^2-4}}+\cdots + \frac{k_N}{2^{(N+1)^2-N^2}}\mid\leq \frac{\mid k_1 \mid}{2^{(N+1)^2-1}}+ \frac{\mid k_2 \mid}{2^{(N+1)^2-4}}+\cdots +\frac{k_N}{2^{(N+1)^2-N^2}}=
\frac{\mid k_1 \mid+\mid k_2\mid 2^3+\cdots \mid k_N \mid 2^{N^2-1}}{2^{(N+1)^2-1}}\leq \frac{2^2+2^32^4+\cdots +2^{(N+1)^2-2}}{2^{(N+1)^2-1}}\leq \frac{\sum_{j=1}^{\infty}2^{(N+1)^2-2j}}{2^{(N+1)^2-1}}=\frac{\frac{2^{(N+1)^2-2}}{\frac{3}{4}}}{2^{(N+1)^2-1}}=\frac{2}{3}$

\end{proof}

\begin{remark}
\noindent This lemma has an important consequence: whenever we have $ \frac{k_1}{2^{(N+1)^2}-1}+\frac{k_2}{2^{(N+1)^2}-4}+\cdots + \frac{k_N}{2^{(N+1)^2-N^2}}+\Z\in\T_m$ then equivalently $ \frac{k_1}{2^{(N+1)^2}-1}+\frac{k_2}{2^{(N+1)^2}-4}+\cdots + \frac{k_N}{2^{(N+1)^2-N^2}}\in[-\frac{1}{4m},\frac{1}{4m}]$.
\end{remark}

\noindent In the previous chapter we observed that $V_{S,m}=\{k\in\Z\mid \frac{k}{2^{n^2}}+\Z\in\T_m\ \forall n\in\N\}$.

\noindent Obviously, the necessary and sufficient conditions for an integer   $k=\sum_{j=0}^Nk_j2^{j^2}$ to be in $V_{S,m}$ are $\frac{\sum_{i=0}^{n-1}k_i2^{i^2}}{2^{n^2}}+\Z\in\T_m\, \forall n\in\N$.

\noindent As we can observe, the admissible values $k_l$ depend on $k_0,\dots,k_{l-1}$. Hence, it is too difficult to find necessary and sufficient conditions for each coefficient.

\noindent We put $k=k_0+k_12^{1^2}+k_22^{2^2}+\cdots$, where $-\frac{2^{(n+1)^2}}{2\cdot 2^{n^2}}\leq k_n\leq\frac{2^{(n+1)^2}}{2\cdot 2^{n^2}}$; or equivalently, $\mid k_n\mid\leq 2^{2n}$.

\noindent We start studying the case $n=1$: $\frac{k}{2^{1^2}}+\Z=\frac{k_0}{2}+\Z\in\T_m$ if and only if $k_0=0$.

\noindent We now study what happens when $n=2$: $\frac{k}{2^{2^2}}+\Z=k_0\frac{1}{2^{2^2}}+k_1\frac{2^{1^2}}{2^{2^2}}+\Z= \frac{k_1}{2^3}+\Z\in\T_m$. We know that $\mid k_1\mid\leq 2^{2}$. Hence, the condition we were searching is that $\frac{k_1}{2^3}+\Z\in\T_m$.

\noindent We move directly to $n=3$. We will obtain a condition for $k_2$.

\noindent $\frac{k}{2^{3^2}}+\Z=k_1\frac{2}{2^9}+k_2\frac{2^4}{2^9}+\Z\in\T_m$. We must find $k_2$ such that the previous holds. $\frac{k_1}{2^8}+\frac{k_2}{2^5}+\Z=\frac{1}{2^5}(\frac{k_1}{2^3}+k_2)+\Z$. At this point we remember that by \ref{dostercios} we have $\mid \frac{k_1}{2^8}+\frac{k_2}{2^5}\mid\leq\frac{2}{3}$; this implies that $\frac{k_1}{2^8}+\frac{k_2}{2^5}\in [-\frac{1}{4m},\frac{1}{4m}]$. Equivalently, $\frac{k_1}{2^3}+k_2\in[-\frac{2^5}{4m},\frac{2^5}{4m}]$. Since $\frac{k_1}{8}\in[-\frac{1}{4m},\frac{1}{4m}]$, we obtain that if $\frac{k_2}{31}+\Z\in\T_m$ (which is equivalent to $k_2\in[-\frac{31}{4m},\frac{31}{4m}]$), then $\frac{k_1}{2^3}+k_2\in[-\frac{2^5}{4m},\frac{2^5}{4m}]$; or equivalently, $\frac{k_1}{2^8}+\frac{k_2}{2^5}+\Z\in\T_m$.

\noindent We consider now $n=4$.

\noindent $\frac{k}{2^{16}}+\Z=k_1\frac{2}{2^{16}}+k_2\frac{2^4}{2^{16}}+k_3\frac{2^9}{2^{16}}+\Z=
\frac{1}{2^7}(\frac{k_1}{2^8}+\frac{k_2}{2^5}+k_3)+\Z$. This number must be in $\T_m$. Once again by \ref{dostercios} $\frac{1}{2^7}(\frac{k_1}{2^8}+\frac{k_2}{2^5}+k_3)+\Z\in\T_m$ if and only if $\frac{1}{2^7}(\frac{k_1}{2^8}+\frac{k_2}{2^5}+k_3)\in[-\frac{1}{4m},\frac{1}{4m}]$; or, equivalently $(\frac{k_1}{2^8}+\frac{k_2}{2^5}+k_3)\in[-\frac{128}{4m},\frac{128}{4m}]$.
We choose $k_3\in[-\frac{127}{4m},\frac{127}{4m}]$. Since $\frac{k_1}{2^8}+\frac{k_2}{2^5}\in[-\frac{1}{4m},\frac{1}{4m}]$, $\frac{k}{2^{16}}+\Z\in\T_m$.

\noindent We can consider the following:

\begin{proposition}
Let $k$ be an integer and let $S=\{\frac{1}{2^{n^2}}+\Z\mid n\in\N\}.$ If $k_0=0$, $\frac{k_1}{8}+\Z\in\T_m$ and $\frac{k_n}{\frac{2^{(n+1)^2}}{2^{n^2}}-1}+\Z=\frac{k_n}{2^{2n+1}-1}+\Z\in\T_m$ for $2\leq n\leq N$, then $\frac{k}{2^{(N+1)^2}}+\Z\in\T_m$. In particular, if $N\geq N(k)$ as defined in \ref{bn}, then $k\in V_{S,m}$
\end{proposition}

\begin{proof}
We will prove this result by induction on $N$.

\noindent We have already seen the proof if $N=2,3$.

\noindent We will now suppose the result true for $N$ and prove it for $N+1$.

\noindent We want $\frac{k}{2^{(N+1)^2}}+\Z= k_1\frac{2}{2^{(N+1)^2}}+ k_2\frac{2^4}{2^{(N+1)^2}}+\cdots +k_N\frac{2^{N^2}}{2^{(N+1)^2}}+\Z =\frac{k_1}{2^{(N+1)^2-1}}+\frac{k_2}{2^{(N+1)^2-4}}+\cdots + \frac{k_N}{2^{(N+1)^2-N^2}}+\Z= \frac{1}{2^{(N+1)^2-N^2}}(\frac{k_1}{2^{N^2-1}}+\frac{k_2}{2^{N^2-4}}+\cdots +k_N)+\Z\in\T_m$.

\noindent By \ref{dostercios}, $\frac{k}{2^{(N+1)^2}}+\Z\in\T_m$ if and only if $\frac{1}{2^{(N+1)^2-N^2}}(\frac{k_1}{2^{N^2-1}}+\frac{k_2}{2^{N^2-4}}+\cdots +k_N)+\Z\in\T_m$.

\noindent That is, we must see that $(\frac{k_1}{2^{N^2-1}}+\frac{k_2}{2^{N^2-4}}+\cdots + k_N)\in[-\frac{2^{(N+1)^2-N^2}}{4m},\frac{2^{(N+1)^2-N^2}}{4m}]$.

\noindent Since the proposition is true for $N$, we know that $\frac{k_1}{2^{N^2-1}}+\cdots +\frac{k_{N-1}}{2^{N^2-(N-1)^2}}\in[-\frac{1}{4m},\frac{1}{4m}]$.

\noindent If furthermore, $\frac{k_N}{2^{(N+1)^2-N^2}-1}+\Z\in\T_m$, or equivalently, $k_N\in[-\frac{2^{(N+1)^2-N^2}-1}{4m},\frac{2^{(N+1)^2-N^2}-1}{4m}]$, then $(\frac{k_1}{2^{N^2-1}})+\frac{k_2}{2^{N^2-4}}+\cdots + k_N\in[-\frac{2^{(N+1)^2-N^2}}{4m},\frac{2^{(N+1)^2-N^2}}{4m}]$.

\noindent This proves the first statement. The second is a direct consequence of the decomposition found at \ref{bn}.

\end{proof}

\begin{remark}

\noindent We have already seen that these conditions are sufficient for an integer $k$ to be in $V_{S,m}$. The next example shows that they are not necessary.

\noindent Take $m=1$, $k=128$. We will see that $k$ does not verify these conditions but $k\in V_{S,m}$.

$128=0\cdot 1+0\cdot 2+8\cdot 16$.

$k_0+\Z=0+\Z\in\T_+$.

$\frac{k_1}{8}+\Z=0+\Z\in\T_+$.

$\frac{k_2}{31}+\Z\approx 0'2580+\Z\notin \T_+$.

\noindent But

$\frac{128}{2}+\Z=0+\Z\in\T_+$.

$\frac{128}{16}+\Z=0+\Z\in\T_+$.

$\frac{128}{512}+\Z=\frac{1}{4}+\Z\in\T_+$.

\noindent Trivially $\frac{128}{2^{k^2}}+\Z\in\T_+$ if $k>3$.

\end{remark}

\section{Characterization of convergent sequences in $\tau_S$. General case}

\noindent Let $(l_j)_{j\in\N}$ be a sequence in $\Z$ converging to $0$ in $\tau_S$. By the proposition in the previous section, we can write $l_j=\sum_{q}k_{j,q}2^{a_q}$, where $(a_n)_{n\in\N}$ a strictly increasing sequence and $a_1=1$.

\noindent Our objective is to find necessary and/or sufficient conditions for the coefficients $k_{j,q}$, in order that $k$ belongs to $V_{S,m}$.

\noindent In the previous chapter, we saw that $l_j\rightarrow 0 \Longleftrightarrow \forall m\in\N,\ \exists j_0$ such that $\forall n\in\N,\ \frac{l_j}{2^{a_n}}+\Z\in\T_m$ for all $j\geq j_m$. 

\noindent Since $l_j=\sum_q k_{j,q}2^{a_q}$, we can write $\frac{l_j}{2^{a_n}}=\sum_q k_{j,q}\frac{2^{a_q}}{2^{a_n}}$.

\noindent We observe that $\sum_q k_{j,q}\frac{2^{a_q}}{2^{a_n}}+\Z=\sum_{q=0}^{n-1} k_{j,q}\frac{2^{a_q}}{2^{a_n}}+\Z,\ \forall n\in\N$.

\noindent What do these conditions mean? 

\noindent For $n=1$, $k_{j,0}\frac{1}{2}+\Z\in\T_m$.

\noindent For $n=2$, $k_{j,0}\frac{1}{2^{a_2}}+k_{j,1}\frac{2}{2^{a_2}}+\Z\in\T_m$

\noindent For $n=3$, $k_{j,0}\frac{1}{2^{a_3}}+k_{j,1}\frac{2}{2^{a_3}}+k_{j,2}\frac{2^{a_2}}{2^{a_3}}+\Z\in\T_m$

\noindent For $n=4$, $k_{j,0}\frac{1}{2^{a_4}}+k_{j,1}\frac{2}{2^{a_4}}+k_{j,2}\frac{2^{a_2}}{2^{a_4}}+ k_{j,3}\frac{2^{a_3}}{2^{a_4}}+\Z\in\T_m$

\noindent We study the first condition, $k_{j,0}\in\{0,1\}$, then $\frac{k_{j,0}}{2}+\Z\in\{0+\Z,\frac{1}{2}+\Z\}$. Since $\frac{k_{j,0}}{2}+\Z\in\T_m$, $k_{j,0}=0$ for all $j>j_1$.

\noindent We turn now to the second one, we choose $j>j_1$, then the condition turns into $\frac{k_{j,1}}{2^{a_2-a_1}}+\Z\in\T_m$. For $m\geq a_2$, and $j\geq j_{a_2}$  $k_{j,1}=0$.

\noindent Repeating the same arguments for each $q$, we choose $m=a_{q+1}$; for all $j\geq j_{a_{q+1}}$, we will get that $k_{j,q}=0$.




\section{Characterization of $V_{S,m}$}

\noindent In this case, we will have that $k=\sum_{i=0}^{\infty}2^{a_i}k_i $ and $\mid k_i\mid \leq \frac{2^{a_{i+1}}}{2\cdot 2^{a_i}}$.

\noindent We will try to reproduce \ref{dostercios}. This time we will need that $a_{n+1}-a_n\geq 2$ forall $n$.

\begin{lemma}\label{dt}

$\mid \frac{k_1}{2^{a_{N+1}-1}}+\frac{k_2}{2^{a_{N+1}-a_2}}+\cdots + \frac{k_N}{2^{a_{N+1}-a_N}}\mid\leq \frac{2}{3}$, for every $N$.
\end{lemma}

\begin{proof}
$\mid \frac{k_1}{2^{a_{N+1}}-1}+\frac{k_2}{2^{a_{N+1}-a_2}}+\cdots + \frac{k_N}{2^{a_{N+1}-a_N}}\mid\leq \frac{\mid k_1 \mid}{2^{a_{N+1}-1}}+ \frac{\mid k_2 \mid}{2^{a_{N+1}-a_2}}+\cdots +\frac{k_N}{2^{a_{N+1}-a_N}}=
\frac{\mid k_1 \mid+\mid k_2\mid 2^{a_2-1}+\cdots \mid k_N \mid 2^{a_N-1}}{2^{a_{N+1}-1}}\leq \frac{\sum_{j=1}^{\infty}2^{a_{N+1}-2j}}{2^{a_{N+1}-1}}=
\frac{\frac{2^{a_{N+1}-2}}{\frac{3}{4}}}{2^{a_{N+1}-1}}=\frac{2}{3}$

\end{proof}

\begin{remark}
\noindent This lemma has an important consequence, whenever we have $ \frac{k_1}{2^{a_{N+1}-1}}+\frac{k_2}{2^{a_{N+1}-a_2}}+\cdots + \frac{k_N}{2^{a_{N+1}-a_N}}+\Z\in\T_m$ equivalently $ \frac{k_1}{2^{a_{N+1}-1}}+\frac{k_2}{2^{a_{N+1}-a_2}}+\cdots + \frac{k_N}{2^{a_{N+1}-a_N}}\in[-\frac{1}{4m},\frac{1}{4m}]$.
\end{remark}

\noindent In the previous chapter we observed that $V_{S,m}=\{k\in\Z\mid \frac{k}{2^{a_n}}+\Z\in\T_m\ \forall n\in\N\}$.

\noindent Obviously, the necessary and sufficient conditions for $k\in\Z$, where   $k=\sum_{j=0}^N\frac{k_j}{2^{a_j}}$ to be in $V_{S,m}$ would be $\frac{\sum_{i=0}^{n-1}k_i2^{a_i}}{2^{a_n}}+\Z\in\T_m\ \forall n\in\N$.

\noindent As we can observe, the admissible values $k_l$ depend on $k_0,\dots,k_{l-1}$. Hence, it is too difficult to find necessary and sufficient conditions for each coefficient.

\noindent We put $k=k_0+k_12^{a_1}+k_22^{a_2}+\cdots$, where $-\frac{2^{a_{n+1}}}{2\cdot 2^{a_n}}\leq k_n\leq\frac{2^{a_{n+1}}}{2\cdot 2^{a_n}}$; or equivalently, $\mid k_n\mid\leq 2^{a_{n+1}-a_n-1}$.

\noindent We start studying the case $n=1$: $\frac{k}{2^{a_1}}+\Z=\frac{k_0}{2}+\Z\in\T_m$ if and only if $k_0=0$.

\noindent We now study what happens when $n=2$: $\frac{k}{2^{a_2}}+\Z=k_0\frac{1}{2^{a_2}}+k_1\frac{2^{a_1}}{2^{a_2}}+\Z= \frac{k_1}{a_3}+\Z\in\T_m$. We know that $\mid k_1\mid\leq 2^{2}$. Hence, the condition we were searching is that $\frac{k_1}{a_3}+\Z\in\T_m$.

\noindent We move directly to $n=3$. We will obtain a condition for $k_2$.

\noindent $\frac{k}{2^{a_3}}+\Z=k_1\frac{2}{2^{a_3}}+k_2\frac{2^{a_2}}{2^{a_3}}+\Z\in\T_m$. We must find $k_2$ such that the previous holds. $\frac{k_1}{2^{a_3-1}}+\frac{k_2}{2^{a_3-a_2}}+\Z=\frac{1}{2^{a_3-a_2}}(\frac{k_1}{2^{a_2-a_1}}+k_2)+\Z$. At this point we remember that by \ref{dt} we have $\mid \frac{k_1}{2^{a_3-1}}+\frac{k_2}{2^{a_3-a_2}}\mid\leq\frac{2}{3}$; this implies that $\frac{k_1}{2^{a_3-1}}+\frac{k_2}{2^{a_3-a_2}}\in [-\frac{1}{4m},\frac{1}{4m}]$. Equivalently, $\frac{k_1}{2^{a_2-1}}+k_2\in[-\frac{2^{a_3-a_2}}{4m},\frac{2^{a_3-a_2}}{4m}]$. Since $\frac{k_1}{2^{a_2-1}}\in[-\frac{1}{4m},\frac{1}{4m}]$, we obtain that if $\frac{k_2}{2^{a_3-a_2}-1}+\Z\in\T_m$ (which is equivalent to $k_2\in[-\frac{2^{a_3-a_2}-1}{4m},\frac{2^{a_3-a_2}-1}{4m}]$), then $\frac{k_1}{2^{a_2-a_1}}+k_2\in[-\frac{2^{a_3-a_2}}{4m},\frac{2^{a_3-a_2}}{4m}]$; or equivalently, $\frac{k_1}{2^{a_2-a_1}}+\frac{k_2}{2^{a_3-a_2}}+\Z\in\T_m$.

\noindent We consider now $n=4$.

\noindent $\frac{k}{2^{a_4}}+\Z=k_1\frac{2}{2^{a_4}}+k_2\frac{2^{a_2}}{2^{a_4}}+k_3\frac{2^{a_3}}{2^{a_4}}+\Z=
\frac{1}{2^{a_4-a_3}}(\frac{k_1}{2^{a_3-1}}+\frac{k_2}{2^{a_3-a_2}}+k_3)+\Z$. This number must be in $\T_m$. Once again by \ref{dt} $\frac{1}{2^{a_4-a_3}}(\frac{k_1}{2^{a_3-1}}+\frac{k_2}{2^{a_3-a_2}}+k_3)+\Z\in\T_m$ if and only if $\frac{1}{2^{a_4-a_3}}(\frac{k_1}{2^{a_3-a_1}}+\frac{k_2}{2^{a_3-a_2}}+k_3)\in[-\frac{1}{4m},\frac{1}{4m}]$; or, equivalently $(\frac{k_1}{2^{a_3-1}}+\frac{k_2}{2^{a_3-a_2}}+k_3)\in[-\frac{2^{a_4-a_3}}{4m},\frac{2^{a_4-a_3}}{4m}]$.
We choose $k_3\in[-\frac{2^{a_4-a_3}-1}{4m},\frac{2^{a_4-a_3}-1}{4m}]$. Since $\frac{k_1}{2^{a_3-1}}+\frac{k_2}{2^{a_3-a_2}}\in[-\frac{1}{4m},\frac{1}{4m}]$, $\frac{k}{2^{a_4}}+\Z\in\T_m$.

\noindent We can consider the following:

\begin{proposition}
If $k_0=0$, $\frac{k_1}{2^{a_2-1}}+\Z\in\T_m$ and $\frac{k_n}{\frac{2^{a_{n+1}}}{2^{a_n}}-1}+\Z=\frac{k_n}{2^{a_{n+1}-a_n}-1}+\Z\in\T_m$ for $2\leq n\leq N$, then $\frac{k}{2^{a_{N+1}}}+\Z\in\T_m$. In particular, if $N\geq N(k)$ defined in \ref{bn}, then $k\in V_{S,m}$
\end{proposition}

\begin{proof}
We will prove this result by induction on $N$.

\noindent We have already seen the proof if $N=2,3$.

\noindent We will now suppose the result true for $N$ and prove it for $N+1$.

\noindent We want $\frac{k}{2^{a_{N+1}}}+\Z= k_1\frac{2}{2^{a_{N+1}}}+ k_2\frac{2^{a_2}}{2^{a_{N+1}}}+\cdots +k_N\frac{2^{a_N}}{2^{a_{N+1}}}+\Z =\frac{k_1}{2^{a_{N+1}-1}}+\frac{k_2}{2^{a_{N+1}-a_2}}+\cdots + \frac{k_N}{2^{a_{N+1}-a_N}}+\Z= \frac{1}{2^{a_{N+1}-a_N}}(\frac{k_1}{2^{a_N-1}}+\frac{k_2}{2^{a_N-a_2}}+\cdots +k_N)+\Z\in\T_m$.

\noindent By \ref{dt}, $\frac{k}{2^{a_{N+1}}}+\Z\in\T_m$ if and only if $\frac{1}{2^{a_{N+1}-a_N}}(\frac{k_1}{2^{a_N-1}}+\frac{k_2}{2^{a_N-a_2}}+\cdots +k_N)+\Z\in \T_m$.

\noindent That is, we must see that $(\frac{k_1}{2^{a_N-1}}+\frac{k_2}{2^{a_N^2-a_2}}+\cdots + k_N)\in[-\frac{2^{a_{N+1}-a_N}}{4m},\frac{2^{a_{N+1}-a_N}}{4m}]$.

\noindent Since the proposition is true for $N$, we know that $\frac{k_1}{2^{a_N-1}}+\cdots +\frac{k_{N-1}}{2^{a_N-a_{N-1}}}\in[-\frac{1}{4m},\frac{1}{4m}]$.

\noindent If furthermore, $\frac{k_N}{2^{a_{N+1}-a_N}-1}\in\T_m$, or equivalently, $k_N\in[-\frac{2^{a_{N+1}-a_N}-1}{4m},\frac{2^{a_{N+1}-a_N}-1}{4m}]$, then $(\frac{k_1}{2^{a_N-1}})+\frac{k_2}{2^{a_N-a_2}}+\cdots + k_N\in[-\frac{2^{a_{N+1}-a_N}}{4m},\frac{2^{a_{N+1}-a_N}}{4m}]$.

\noindent This proves the first statement. The second is a direct consequence of the decomposition found at \ref{bn}.

\end{proof}

\section{General results}

\noindent We start the section proving the following algebraic result.

\begin{proposition}\label{S=prufer}
Let $(a_n)_{n\in\N}$ be a strictly increasing sequence and let $S=\{2^{-a_n}+\Z\mid n\in\N\}$.

\noindent Then $\langle S\rangle =\Z(2^\infty)$.
 
\end{proposition}

\begin{proof}
\noindent Let $S_0=\{2^{-n}+\Z\mid n\in\N \}$. We know that $\Z(2^\infty)=\langle S_0 \rangle$.

\noindent Since $S\subset S_0$, $\langle S\rangle \subset \langle S_0 \rangle=\Z(2^\infty)$.

\noindent Conversely, let $x+\Z\in \Z(2^\infty)$. $x=\frac{k}{2^n}$ for some $k\in\Z$ and some $n\in\N$.

\noindent Since $(a_n)$ is strictly increasing $a_n\geq n$.

\noindent Hence, $x=\frac{k}{2^n}=\frac{k2^{a_n-n}}{2^{a_n}}$. Thus $x\in\langle S\rangle$.

\end{proof}

\noindent We now recover our interest in convergent sequences in $\tau_S$.

\begin{theorem}
Let $S=\{2^{-n^2}+\Z\mid n\in\N\}$. Then $\tau_2\neq \tau_S$.
\end{theorem}

\begin{proof}
It suffices to prove that there exists a sequence $l_j\rightarrow 0$ in $\tau_2$ such that $l_j\nrightarrow 0$ in $\tau_S$.

\noindent We use the criteria found in the previous chapter.

\noindent Fix $l_j=2^j$.

\noindent Obviously, $l_j\rightarrow 0$ in $\tau_2$.

\noindent We observed that, $l_j\rightarrow 0$ in $\tau_S\Longleftrightarrow \forall m\in\N\, \exists j_m$ such that $\forall n$ $\frac{l_j}{2^{n^2}}+\Z\in\T_m$ for all $j\geq j_m$.

\noindent We now want to see that $l_j\nrightarrow 0$ in $\tau_S$

\noindent Fix $m$. It suffices to show that $A=\{j_n\in\N\mid \exists n$ such that $\frac{l_j}{2^{n^2}}+\Z\notin\T_m$ if $j\geq j_m\}$ is cofinal in $\N$.

\noindent Fix $n$. Find $j_n$ such that $\frac{l_{j_n}}{2^{n^2}}=\frac{1}{2}\Longleftrightarrow l_{j_n}=2^{n^2-1}\Longleftrightarrow j_n=n^2-1$.

\noindent Thus, $A=\{n^2-1\mid n>1\}$ and it is obviously cofinal.

\noindent Hence $l_j\nrightarrow 0$ in $\tau_S$ and $\tau_2\neq\tau_S$.

\end{proof}

\noindent Can we generalize this result? The answer is affirmative

\begin{theorem}\label{tau2noestauS}
Let $(a_n)$ be a strictly increasing sequence and let $S=\{\frac{1}{2^{a_n}}+\Z\mid n\in\N\}$. Then $\tau_2\neq \tau_S$.
\end{theorem}

\begin{proof}
As above it suffices to show the existence of a sequence $l_j\rightarrow 0$ in $\tau_2$ such that $l_j\nrightarrow 0$ in $\tau_S$.

\noindent Fix $l_j=2^j$.

\noindent As before, $l_j\rightarrow 0$ in $\tau_2$.

\noindent We observed that, $l_j\rightarrow 0$ in $\tau_S\Longleftrightarrow \forall m\in\N\, \exists j_m$ such that $\forall n$ $\frac{l_j}{2^{a_n}}+\Z\in\T_m$ if $j\geq j_m$.

\noindent Fix $m=1$. It suffices to show that $B=\{j_n\in\N\mid \exists n$ such that $\frac{l_j}{2^{a_n}}+\Z\notin\T_+$ if $j\geq j_1\}$ is cofinal in $\N$.

\noindent Fix $n$. Find $j_n$ such that $\frac{l_{j_n}}{2^{a_n}}=\frac{1}{2}\Longleftrightarrow l_{j_n}=2^{a_n-1}\Longleftrightarrow j_n=a_n-1$.

\noindent Thus, $B=\{a_n-1\mid n>1\}$. Since $(a_n)$ is strictly increasing, B is cofinal.

\noindent Hence $l_j\nrightarrow 0$ in $\tau_S$ and $\tau_2\neq\tau_S$.

\end{proof}

\noindent Next question would be:

\begin{question}

Fix $(a_n)\neq(b_n)$ strictly increasing sequences. $S_1=\{2^{-a_n}+\Z\mid n\in n\}$, $S_2=\{2^{-b_n}+\Z\mid n\in\N\}$. Is $\tau_{S_1}\neq\tau_{S_2}$?

\end{question}


\noindent We now compare two topologies on $\Z$. On the one hand, $\tau_S$, the topology of uniform convergence on $S=\{\frac{1}{2^{a_n}}+\Z\mid n \in \N\}\subset\T$, and on the other hand the linear group topology generated by the sequence $\{2^{a_n}\mid n\in\N\}$, whose neighbourhood basis of zero is given by $\{2^{a_n}\Z\mid n\in\N\}$

\begin{question}[Question 1]
Let $(a_n)_{n\in\N}$ be an increasing sequence. Then $\{2^{a_n}\Z\mid n\in\N\}$ 
forms a neighbourhood basis for a Haussdorff group topology on $\Z$.
\end{question}

\begin{question} [Question 2]
Let $(a_n)_{n\in\N}$ be an increasing sequence. Let $S=\{\frac{1}{2^{a_n}}+\Z\mid n\in\N\}$. Then $\{2^{a_n}\Z\mid n\in\N\}$ is a neighbourhood basis for $(\Z,\tau_S)$.
\end{question}

\noindent We shall see in this section the answer to both questions. The first one is true, but the resulting topology is the $2-adic$ topology. Unfortunately, the second one is false.

\begin{proposition}\label{tauS=<S>}
If $\{2^{a_n}\Z\mid n\in\N\}$ is a neighbourhood basis for $\tau_S$, then $(\Z,\tau_S)^\wedge=\langle S\rangle$. By proposition \ref{S=prufer}, this would mean that $(\Z,\tau_S)^\wedge=\Z(2^\infty)$.

\end{proposition}

\begin{proof}
 In \ref{3.5.2} we have already seen that $\Z(2^\infty)\subset(\Z,\tau_S)^\wedge$.

\noindent Let $\chi\in (\Z,\tau_S)^\wedge$.Since $\chi$ is continuous, 
there exists a neighbourhood $U=2^{a_n}\Z$ such that $\chi(U)=\chi(2^{a_n}\Z)\subset \T_+$. Now $2^{a_n}\Z$ is a subgroup of $\Z$, hence, its image by a homomorphism will be again a subgroup in $\T$. The only subgroup of $\T$ contained in $\T_+$ is $\{0\}$. Hence, $\chi(2^{a_n}\Z)=\{0+\Z\}$. Let $x+\Z=\chi(1)$. Hence, $\chi(2^{a_n})=2^{a_n}x+\Z=0+\Z$.
This implies that $2^{a_n}x\in \Z$.

\noindent Hence: $(\Z,\tau_S)^\wedge \subset\{\frac{k}{2^{a_n}}+\Z\mid \  k\in\Z, n\in\N\}=\langle S\rangle=\Z(2^\infty).$

\end{proof}

\begin{proof}[Question 2 is false if $(a_{n+1}-a_n)$ bounded]
Suppose $S=\{2^{-n}+\Z\}$ and that question 2 is true.

\noindent We have seen in \ref{tau_Sesdiscreta} that in this case, $\tau_S=\tau_{dis}$.

\noindent Hence $(\Z,\tau_S)=(\Z,\tau_{dis})$. We dualize.

\noindent By \ref{tauS=<S>} and \ref{S=prufer} $(\Z,\tau_S)^\wedge=\langle S \rangle=\Z(2^\infty)$. 
On the other hand, $(\Z,\tau_{dis})^\wedge=\T$. This leads to the contradiction  $\Z(2^\infty)=\T$.

\end{proof}

\begin{proposition}
Let $(a_n)$ as in \ref{bn}, $\mathcal{U}=(2^{a_n}\Z)$ and $\mathcal{V}=(2^n\Z)$. The $\tau_\mathcal{U}=\tau_\mathcal{V}$
\end{proposition}

\begin{proof}[proposition + question 1]
$\mathcal{U}\subseteq\mathcal{V}$, hence $\tau_\mathcal{U}\leq\tau_\mathcal{V}$.

\noindent Fix now $U\in\mathcal{U}$, $U=2^{a_{n_0}}\Z$. Since $(a_n)$ is increasing, $V=2^{a_{n_0}}\subseteq U$.  Hence $\mathcal{U}\supseteq\mathcal{V}$.

\noindent Since $\mathcal{U}$ is a neighbourhood basis for the $2-adic$ topology (which is Hausdorff), we have, also, proven question 1.
\end{proof}

\begin{proof} [Question 2 is false]
Suppose it is true, then $\tau_S$ has $(2^{b_n}\Z)$ as neighbourhood basis. By previous proposition $\tau_S=\tau_2$, which we know is false by \ref{tau2noestauS}

\end{proof}

\chapter{Generalization for $S$-topologies}
\noindent In chapter $3$ we focused our interest on the $2$-adic topology and on $S$-topologies generated by $S=\{2^{-a_n}+\Z\mid n\in\N\}$, which refine the $2$-adic topology.

\noindent In this chapter we shall generalize these topologies into a wider frame.

\noindent We shall define new topologies and obtain some general results.

\section{Linear group topologies on the integers}

\begin{proposition}
Let $(b_n)_{n\in\N}$ be a sequence as in $\ref{bn}$, then $\mathcal{U}=\{b_n\Z\mid n\in\N\}$ is a neighbourhood basis at $0$ for a Hausdorff topology.
\end{proposition}

\begin{proof}
We should check conditions (i)-(iv) in \ref{ctg}.


\noindent (i) It is trivial that $0\in b_n\Z$ for all $n\in\N$.

\noindent (ii) We shall prove that $b_n\Z$ is symmetric. Let $k\in b_n\Z$. By definition $k=b_nz$ for some $z\in\Z$. Since $\Z$ is symmetric, $-z\in\Z$ and $-k=b_n(-z)\in b_n\Z$.

\noindent (iii) We shall prove that $b_n\Z+b_n\Z=b_n\Z$. Obviously, $b_n\Z+b_n\Z\supseteq b_n\Z$. 

\noindent Let $a,b\in b_n\Z$, we must check $a+b\in b_n\Z$. By definition, $a=b_nk_1;\ b=b_nk_2$. Hence, $a+b=b_n(k_1+k_2)$. Since $k_1+k_2\in\Z$, we get $a+b\in b_n\Z$.

\noindent (iv) Let $b_n\Z,b_m\Z\in\mathcal{U}$, and suppose $n\leq m$. 
We shall prove $b_n\Z\cap b_m\Z\supseteq b_m\Z$. 
It clearly suffices to show that $b_m\Z\subseteq b_n\Z$. Let $x\in b_m\Z$, then 
$x=b_mz=\frac{b_m}{b_n}b_nz$. By definition of $(b_n)$ and, since $m\geq n$, $\frac{b_m}{b_n}\in\Z$. Hence $x\in b_n\Z=b_n\Z$.

\noindent In order to prove that the resulting topology is Hausdorff, we must see that $\cap_{n\in\N}b_n\Z=\{0\}$.

\noindent Obviously $\cap_{n\in\N}b_n\Z\supseteq \{0\}$.

\noindent $x\in b_n\Z \Longleftrightarrow x \mbox{ is multiple of } b_n$. Hence $x\in\cap b_n\Z\Longleftrightarrow b_n\mid x\mbox{ for all } n\Rightarrow x=0$

\end{proof}

\begin{notation}
Let $(\Z,\tau_{(b_n\Z)})$ be the group of the integers endowed with the topology defined by the neighbourhood basis $\mathcal{U}=\{b_n\Z\mid n\in\N\}$.

\end{notation}

\begin{proposition}
$(\Z,\tau_{(b_n\Z)})^\wedge=\{\frac{k}{b_n}+\Z\mid n\in\N, k\in\Z\}$.
\end{proposition}

\begin{proof}
First we show that $(\Z,\tau_{(b_n\Z)})^\wedge\subseteq\{\frac{k}{b_n}+ºZ\mid n\in\N, k\in\Z\}$ holds. 
Let $\chi\in(\Z,\tau_{(b_n\Z)})^\wedge$. Since $\chi$ is continuous, 
there exists a neighbourhood $b_n\Z$ such that $\chi(b_n\Z)\subset\T_+$. Hence, $\chi(b_n\Z)$ is a subgroup contained in $\T_+$; therefore $\chi(b_n\Z)=\{0+\Z\}$.

\noindent Let $x+\Z=\chi(1)\in\T$, $\chi(b_n)=b_nx+\Z=0+\Z$. 
Hence $b_nx\in \Z\Longleftrightarrow \mbox{ there exists } k\in\Z\mbox{ such that }x=\frac{k}{b_n}$.

\noindent In conclussion: $(\Z,\tau_{(b_n\Z)})^\wedge\subseteq\{\frac{k}{b_n}+\Z\mid n\in\N, k\in\Z\}$

\noindent In order to prove that $(\Z,\tau_{(b_n\Z)})^\wedge\supseteq\{\frac{k}{b_n}+\Z\mid n\in\N, k\in\Z\}$; 
let $k\in\Z, n\in\N$. 

\noindent Observe that $\ker(\frac{k}{b_n}+\Z)= \{j\in\Z\mid j\frac{k}{b_n}+\Z=0+\Z \}\supseteq b_n\Z$.

\noindent By \ref{subgrupoabierto}, $\ker(\frac{k}{b_n}+\Z)$ is open.

\noindent Since $(\frac{k}{b_n}+\Z)^{-1}(\T_n)\supset \ker (\frac{k}{b_n}+\Z)$ our homomorphism is continuous.

\end{proof}

\noindent Let $(l_j)\subset\Z$ be a sequence. We want to find a criterion for $l_j\rightarrow 0$ in $\tau_{(b_n\Z)}$.

\begin{proposition}
$l_j\rightarrow 0 \mbox{ in }\tau_{(b_n\Z)}\Longleftrightarrow \mbox{ For all } n\in\N\mbox{ there exists } j_n\mbox{ such that } b_n\mid l_j\mbox{ if } j\geq j_n$
\end{proposition}

\begin{proof}
Let $l_j\rightarrow 0\mbox{ in }\tau_{(b_n\Z)}$;
 by definition of convergence, for every $n\in\N$, there exists $j_n$ such that $l_j\in b_n\Z$ if $j\geq j_n$; or equivalently: 
$l_j\rightarrow 0 \mbox{ in }\tau_{(b_n\Z)}\Longleftrightarrow \mbox{ for all } n\in\N\mbox{ there exists } 
j_n\mbox{ such that } b_n\mid l_j\mbox{ if } j\geq j_n$.

\end{proof}

\section{Generalization of $S$-topologies}

\begin{notation}
\noindent Let $(b_n)_{n\in\N_0}$ as in \ref{bn}. Define $S=\{\frac{1}{b_n}+\Z\mid n\in\N\}$.

\noindent Then $V_{S,m}=\{k\in\Z\mid \mbox{ for all } z+\Z\in S\ zk+\Z\in\T_m\}=\{k\in\Z\mid \frac{k}{b_n}+\Z\in\T_m \mbox{ for all } n\in\N\}$
\end{notation}

\noindent We look for a characterization of sequences converging to $0$ in $\tau_S$.

\begin{proposition}
$l_j\rightarrow 0\mbox{ in }\tau_S\Longleftrightarrow \mbox{ for all } m\in\N\mbox{ there exists }j_m\mbox{ such that }\frac{l_j}{b_n}+\Z\in\T_m\mbox{ for all }n\in\N \mbox{ and } j\geq j_m$
\end{proposition}

\begin{proof}
Fix $l_j\rightarrow 0$ in $\tau_S$.

\noindent By definition (of convergence) for all $m\in\N$, there exists $j_m$ such that $l_j\in V_{S,m}$ if $j\geq j_m$.

\noindent $l_j\in V_{S,m}\Longleftrightarrow\frac{l_j}{b_n}\in\T_m$ for all $n\in\N$.

\noindent Combining both statements we get: $l_j\rightarrow 0\mbox{ in }\tau_S\Longleftrightarrow \mbox{ for all } m\in\N\mbox{ there exists }j_m\mbox{ such that }\frac{l_j}{b_n}+\Z\in\T_m\mbox{ for all }n\in\N \mbox{ if } j\geq j_m$.

\end{proof}

\begin{proposition}
Let $(b_n)_{n\in\N_0}$ as in \ref{bn}. Put $S=\{\frac{1}{b_n}+\Z\mid n\in\N\}$ and $(l_j)\subset\Z$.

\noindent Then $l_j\rightarrow 0$ in $\tau_S\Rightarrow l_j\rightarrow 0$ in $\tau_{(b_n\Z)}$.
\end{proposition}

\begin{proof}
Let $n_0\in\N$. Fix $m=b_{n_0}$.

\noindent By hypothesis, there exists $j_m$ such that $\frac{l_j}{b_n}+\Z\in\T_m$ for all $n$ and $j\geq j_m$.

\noindent This means that $\frac{l_j}{b_n}+\Z,\dots,\frac{ml_j}{b_n}+\Z\in\T_+\Longleftrightarrow\frac{l_j}{b_n}+\Z, \dots,\frac{b_{n_0}l_j}{b_n}+\Z\in\T_+$ for all $n$.

\noindent In particular this is true when $n=n_0$; then $\frac{l_j}{b_{n_0}}+\Z,\dots,\frac{b_{n_0}l_j}{b_{n_0}}+\Z\in\T_+$.

\noindent But $\frac{l_j}{b_{n_0}}+\Z,\dots,\frac{b_{n_0}l_j}{b_{n_0}}+\Z=\langle\frac{l_j}{b_{n_0}}+\Z\rangle $, which is a subgroup contained in $\T_+$. 

\noindent Hence $\frac{l_j}{b_{n_0}}+\Z=0+\Z$. Or, equivalently, $b_{n_0}\mid l_j$ if $j\geq j_m$. That is, $l_j\rightarrow 0$ in $\tau_{(b_n\Z)}$.
 
\end{proof}

\begin{remark}
The previous proposition implies that 

\centerline{$\id:(\Z,\tau_S)\rightarrow(\Z,\tau_{(b_n\Z)})$}

\noindent is continuous. Since it is also surjective, the mapping:

\centerline{$\id^{\wedge}:(\Z,\tau_{(b_n\Z)})^\wedge\rightarrow(\Z,\tau_S)^\wedge$} 

\noindent is injective, and $(\Z,\tau_{(b_n\Z)})^\wedge\subseteq(\Z,\tau_S)^\wedge$

\end{remark}

\section{Characterization of sequences converging to $0$ in $\tau_S$}

\noindent By \ref{bn}, we can write $l_j=\sum_{s=0}^{N(l_j)}b_sk_{s,j}$. We want to find information about $k_{s,j}$ in order that $(l_j)$ is convergent to $0$ in $\tau_S$.

\begin{proposition}
Let $(l_j)$ be a null sequence in $\tau_S$. Let $k_{q,j}$ be the coefficients introduced above. For any $s$ there exists $j_q$ such that $k_{q,j}=0$ if $j\geq j_q$.
\end{proposition}

\begin{proof}
We prove this result by induction.

\noindent We begin by proving the result for $q=0$. Since $l_j\rightarrow 0$ in $\tau_{(b_n\Z)}$ , there exists $j_0\in\N$ such that $b_1\mid l_j$ for all $j\geq j_0$.

\noindent $l_j=\sum_{s=0}^{N(l_j)}b_sk_{s,j}$. Then  $\frac{l_j}{b_1}=\frac{b_0}{b_1}k_{0,j}+\sum_{s=1}^{N(l_j)}\frac{b_s}{b_1}k_{s,j}$. Since, $\frac{b_s}{b_1}\in\Z$ for $s\geq 1$, we get that  $\sum_{s=1}^{N(l_j)}\frac{b_s}{b_1}k_{s,j}\in\Z$ and hence, $\frac{b_0}{b_1}k_{0,j}\in\Z$, $\forall j\geq j_0$. Since $\mid k_{0,j}\mid\leq\frac{b_1}{2b_0}$, we get $\mid \frac{b_0}{b_1}k_{0,j}\mid\leq\frac{1}{2}$. Hence $\frac{b_0}{b_1}k_{0,j}\in\Z$ if and only if $k_{0,j}=0$.

\noindent As conclusion, $k_{0,j}=0$ if $j\geq j_0$.

\noindent Let us prove the inductive step; that is, we suppose that there exist $j_0\leq j_1\leq\cdots\leq j_q$ such that $k_{q,j}=0$ for all $j\geq j_q$.

\noindent By definition, there exists $j_{q+1}\geq j_q$ such that $b_{q+2}\mid l_j$ for all $j\geq j_{q+1}$ (indexes are different from the ones in previous section).

\noindent Now, $l_j=\sum_{s=0}^{N(l_j)}b_sk_{s,j}= \sum_{s=q+1}^{N(l_j)}b_sk_{s,j}$ for $j\geq j_{q+1}$; the last equality holds by the inductive hypothesis.

\noindent $\frac{l_j}{b_{q+2}}= \frac{b_{q+1}}{b_{q+2}}k_{q+1,j}+\sum_{s=q+2}^{N(l_j)}\frac{b_sk_{s,j}}{b_{q+2}}$.

\noindent Since $\frac{l_j}{b_{q+2}}\in\Z$ for $j\geq j_{q+1}$, $\sum_{s=q+2}^{N(l_j)}\frac{b_sk_{s,j}}{b_{q+2}}\in\Z$ and the equality holds, we deduce that $\frac{b_{q+1}}{b_{q+2}}k_{q+1,j}\in\Z$.

\noindent As before, $\mid \frac{b_{q+1}}{b_{q+2}}k_{q+1,j}\mid\leq\frac{1}{2}$. This implies that $\frac{b_{q+1}}{b_{q+2}}k_{q+1,j}\in\Z$ if and only if $k_{q+1,j}=0$. And the result follows.

\end{proof}

\section{Characterization of the elements of $V_{S,m}$}

\noindent In this section, we consider $(b_n)_{n\in\N_0}$ as in \ref{bn}; $S=\{\frac{1}{b_n}+\Z\mid n\in\N\}\subset\T$, and, consequently, $V_{S,m}=\{k\in\Z\mid \frac{k}{b_n}\in\T_m\mbox{ for all }n\}$.

\noindent First, we shall find a characterization for $k$ to be in $V_{S,m}$. Later, we try to find weaker (but on the other hand easier) sufficient or necessary conditions for $k$ to be in $V_{S,m}$.

\begin{theorem}\label{caracterizacion}
Let $k=\sum_{s=0}^{N(k)}k_sb_s$ be an integer, and let $k_0,\dots,k_{N(k)}$ be the coefficients obtained in \ref{bn}. Then $k\in V_{S,m}\Longleftrightarrow \sum_{s=0}^{n-1}\frac{k_sb_s}{b_n}\Z\in\T_m$ for all $n\in\N$.
\end{theorem}

\begin{proof}
By \ref{bn} we write $k=\sum_{s=0}^{N(k)}k_sb_s$.

\noindent Then, $\frac{k}{b_n}=\sum_{s=0}^{N(k)}\frac{k_sb_s}{b_n}$, and $\frac{k}{b_n}+\Z=\sum_{s=0}^{N(k)}\frac{k_sb_s}{b_n}+\Z$.

\noindent Since $b_n\mid b_{m}$ if $m\geq n$ and $k_m\in\Z$, since $\frac{k}{b_n}+\Z\in\T_m$ the assertion follows.

\end{proof}

\begin{lemma}\label{unmedio}
$\mid\sum_{s=0}^{n-1}\frac{k_sb_s}{b_n}\mid\leq\frac{1}{2}$ for all $n\in\N$
\end{lemma}

\begin{proof}
$\mid\sum_{s=0}^{n-1}\frac{k_sb_s}{b_n}\mid=\frac{\mid\sum_{s=0}^{n-1}k_sb_s\mid}{b_n} \stackrel{\ref{bn}}{\leq} \frac{\frac{b_n}{2}}{b_n}=\frac{1}{2}$.

\end{proof}

\noindent The following corollary is the characterization we were seeking.

\begin{corollary} \label{car}
$k\in V_{S,m}\Longleftrightarrow \mid \sum_{s=0}^{n-1}\frac{k_sb_s}{b_n}   \mid\leq\frac{1}{4m}$ for all $n\in\N_0$.
\end{corollary}

\begin{proof}
The proof is an inmediate consequence of \ref{caracterizacion} and \ref{unmedio}.
\end{proof}

\begin{proposition}
$k\in V_{S,m}$ implies $\mid\frac{k_nb_n}{b_{n+1}}\mid\leq\frac{3}{8m}$ for all $n\in\N_0$.
\end{proposition}

\begin{proof}
For $n=0$, we have $\mid\frac{k_0b_0}{b_1}\mid\stackrel{\ref{car}}{\leq}\frac{1}{4m}\leq\frac{3}{8m}$

\noindent Since $k\in V_{S,m}$, by \ref{car}, $\mid\sum_{s=0}^{n-1}\frac{k_sb_s}{b_n}\mid\leq\frac{1}{4m}$, for all $n\in\N$.

\noindent We suppose that for some $n\geq 1$, we have $\frac{b_n\mid k_n\mid}{b_{n+1}}>\frac{3}{8m}$.

\noindent $\mid\sum_{s=0}^{n-1}\frac{k_sb_s}{b_{n+1}}\mid= \mid\sum_{s=0}^{n}\frac{k_sb_s}{b_{n+1}}-\frac{k_nb_n}{b_{n+1}}\mid\stackrel{\ref{car}}{>}  \frac{3}{8m}-\frac{1}{4m}=\frac{1}{8m}$.

\noindent $\frac{\mid\sum_{s=0}^{n-1}k_sb_s\mid}{b_n}= \frac{\mid\sum_{s=0}^{n-1}k_sb_s\mid}{b_{n+1}}\frac{b_{n+1}}{b_n} >\frac{1}{8m}2=\frac{1}{4m}$; which contradicts the characterization in \ref{car}.

\end{proof}

\begin{proposition}
If $\mid\frac{k_nb_n}{b_{n+1}}\mid\leq\frac{1}{8m}$ for $0\leq n\leq N$ then $k=\sum_{n=0}^{N}b_nk_n\in V_{S,m}$
\end{proposition}

\begin{proof}
Let $0\leq n \leq N$, then $\mid\sum_{p=0}^{n-1}\frac{k_pb_p}{b_n}\mid\leq\sum_{p=0}^{n-1}\mid\frac{k_pb_p}{b_n}\mid= \sum_{p=0}^{n-1}\frac{b_{p+1}}{b_n}\frac{\mid k_p\mid b_p}{b_{p+1}}\leq \frac{1}{8m}\sum_{p=0}^{n-1}\frac{b_{p+1}}{b_n}$.

\noindent Since $\frac{b_{n-1}}{b_n}\leq\frac{1}{2}$, $\frac{b_{n-k}}{b_n}\leq\frac{1}{2^k}$.

\noindent In the first equality, we consider the change $k=n-1-p$. $\sum_{p=0}^{n-1}\frac{b_{p+1}}{b_n}=\sum_{k=0}^{n-1}\frac{b_{n-k}}{b_n}\leq \sum_{k=0}^{n-1}\frac{1}{2^k}\leq\sum_{k=0}^{\infty}2^{-k}=2$.

\noindent Hence, $\mid\sum_{p=0}^{n-1}\frac{k_pb_p}{b_n}\mid\leq \frac{1}{8m}\sum_{p=0}^{n-1}\frac{b_{p+1}}{b_n}\leq\frac{1}{4m}$.

\noindent Hence $k\in V_{S,m}$.

\end{proof}

\section{Relation between linear topologies and $S$-topologies associated}

\noindent In this section we will prove some differences between both types of topologies.

\noindent Let $\mathbf{b}$ $=(b_n)_{n\in\N}$ be a sequence as in \ref{bn}. We define the linear topology associated to $\mathbf{b}$ as the one having $\{b_n\Z\mid n\in\N\}$ as neighbourhood basis at $0$. We shall denote it by $\tau_{(b_n\Z)}$. We shall also deal with the $S$-topology where $S=\{\frac{1}{b_n}+\Z\mid n\in\N\}$.

\noindent We already know some facts of each topology; in this section we will prove that $\tau_S\neq\tau_{(b_n\Z)}$ for any $\mathbf{b}$.

\noindent In \ref{tau_Sesdiscreta} we have already seen that $\tau_S\neq\tau_{(b_n\Z)}$ if $(\frac{b_{n+1}}{b_n})$ is bounded ($\tau_{(b_n\Z)}$ is never discrete). 

\noindent As a consequence we must only prove the result if ($\frac{b_{n+1}}{b_n}$) is unbounded.

\begin{proposition}
Let $\mathbf{b}$ $=(b_n)$ be a sequence as in \ref{bn}, and such that ($\frac{b_{n+1}}{b_n}$) is unbounded. Let $\tau_{(b_n\Z)}$ and $\tau_S$ the linear and the $S$-topology associated to $\mathbf{b}$, respectively. Then $\tau_S\neq\tau_\mathfrak{L}$.
\end{proposition}

\begin{proof}
Since ($\frac{b_{n+1}}{b_n}$) is unbounded, there exists a sequence $(n_k)_{k\in\N}\subset\N$ such that $\frac{b_{n_k}}{b_{n_k+1}}\rightarrow 0$. 

\noindent It suffices to see that there exists a sequence $(l_j)$ such that $l_j\rightarrow 0$ in $\tau_{b_n\Z}$, but $l_j\nrightarrow 0$ in $\tau_S$.

\noindent Fix $l_j=b_j[\frac{b_{j+1}}{2b_j}]$, where $[x]$ is the biggest integer which is smaller or equal to $x$.

\noindent Since $b_j\mid l_j$ it is obvious that $l_j\rightarrow 0$ in $\tau_\mathfrak{L}$.

\noindent Let us show that $l_j\nrightarrow 0$ in $\tau_S$. 
Since $(\frac{b_{n_k}}{b_{n_k+1}})_k\rightarrow 0$, there exists $k_0$ such that $\frac{b_{n_k}}{b_{n_k+1}}\leq\frac{1}{16}$ if $k\geq k_0$.

\noindent We prove that for $A:=\{n_k\mid k\geq k_0\}$ and for every $j=n_{k_0}$, there exists $n\in\N$ (we prove that $n=n_{k_0}+1$) such that $\frac{l_j}{b_n}\Z\notin\T_+$, which shows that $l_j\nrightarrow 0$ in $\tau_S$.

\noindent $\frac{l_j}{b_{j+1}}+\Z=\frac{b_j}{b_{j+1}}[\frac{b_{j+1}}{2b_j}]+\Z$.

\noindent If $\frac{b_{j+1}}{b_j}$ is even, then $[\frac{b_{j+1}}{2b_j}]=\frac{b_{j+1}}{2b_j}$, and $\frac{l_j}{b_{j+1}}+\Z=\frac{b_j}{b_{j+1}} \frac{b_{j+1}}{2b_j}+\Z =\frac{1}{2}+\Z\notin\T_+$.

\noindent If $\frac{b_{j+1}}{b_j}$ is odd then $[\frac{b_{j+1}}{2b_j}]=\frac{b_{j+1}}{2b_j}-\frac{1}{2}$, and $\frac{l_j}{b_{j+1}}+\Z=\frac{b_j}{b_{j+1}}(\frac{b_{j+1}}{2b_j}-\frac{1}{2})+\Z= \frac{1}{2}(1-\frac{b_j}{b_{j+1}})$. By the choice of $k_0$, $\frac{l_j}{b_{j+1}}+\Z\notin\T_+$.

\noindent The result follows.

\end{proof}

\bibliographystyle{plain}

}
\backmatter

\end{document}